\documentclass[11pt]{article}

\usepackage[top=1in, bottom=1.25in, left=0.75in, right=0.75in]{geometry}
\usepackage{hyperref}
\usepackage{amssymb}
\usepackage[linesnumbered,lined,ruled]{algorithm2e}
\usepackage{amsmath,amsthm,amsfonts,dsfont,mathtools,bm}
\usepackage{url}
\usepackage{rotating}
\usepackage{multirow}
\usepackage{color}
\usepackage{xcolor}
\usepackage{enumitem}
\DeclareSymbolFont{rsfs}{U}{rsfs}{m}{n}
\DeclareSymbolFontAlphabet{\mathscrsfs}{rsfs}

\hypersetup{
  colorlinks   = true, %Colours links instead of ugly boxes
  urlcolor     = blue, %Colour for external hyperlinks
  linkcolor    = blue, %Colour of internal links
  citecolor   = blue %Colour of citations
}

\newtheorem{theorem}{Theorem}
\newtheorem{definition}[theorem]{Definition}

\newtheoremstyle{myremark} % name
    {\topsep}                    % Space above
    {\topsep}                    % Space below
    {\rm}                        % Body font
    {}                           % Indent amount
    {\bf}                        % Theorem head font
    {.}                          % Punctuation after theorem head
    {.5em}                       % Space after theorem head
    {}  % Theorem head spec (can be left empty, meaning normal)

\newtheorem{lemma}[theorem]{Lemma}

\newtheorem{corollary}[theorem]{Corollary}

\newtheorem{proposition}[theorem]{Proposition}

\theoremstyle{myremark}
\newtheorem{remark}{Remark}[section]

\renewcommand{\P}{\operatorname{\mathbb{P}}}
\newcommand{\Q}{\operatorname{\mathbb{Q}}}
\newcommand{\W}{\operatorname{\mathbb{W}}}
\newcommand{\E}{\operatorname{\mathbb{E}}}

\newcommand{\R}{\mathbb{R}}

\newcommand{\T}{T}

\newcommand{\D}{\boldsymbol{D}}
\newcommand{\Dm}{\D(\m)}
\newcommand{\F}{\mathcal{F}}
\newcommand{\rmd}{\mathrm{d}}

\DeclareMathOperator{\cov}{\mathrm{cov}}

\renewcommand{\hat}{\widehat}

\DeclareMathOperator*{\argmin}{arg\,min}

\newcommand{\bA}{\bm{A}}
\newcommand{\bB}{\bm{B}}
\newcommand{\bI}{\bm{I}}
\newcommand{\bY}{\bm{Y}}
\newcommand{\bW}{\bm{W}}

\newcommand{\cF}{{\mathcal F}}

\newcommand{\wtSigma}{\widetilde{\Sigma}}
\newcommand{\wtbG}{\widetilde{\boldsymbol G}}
\newcommand{\wtG}{\widetilde{G}}
\newcommand{\wtsigma}{\widetilde{\sigma}}
\newcommand{\wtgamma}{\widetilde{\gamma}}
\newcommand{\wtbeta}{\widetilde{\beta}}

\newcommand{\wtbz}{\widetilde{\bz}}
\newcommand{\wtbu}{\widetilde{\bu}}
\newcommand{\wtbv}{\widetilde{\bv}}
\newcommand{\wtm}{\widetilde{\m}}

\newcommand{\Unif}{{\sf Unif}}
\newcommand{\onu}{\overline{\nu}}

\newcommand{\salg}{\mbox{\rm\tiny alg}}

\newcommand{\sTAP}{\mbox{\rm\tiny TAP}}
\newcommand{\eps}{\varepsilon}
\newcommand{\n}{\boldsymbol{n}}
\newcommand{\kalg}{k_{\salg}}
\DeclareMathOperator{\atanh}{arctanh}

\newcommand{\Dh}{D_{-\bh}}

\newcommand{\pl}{{\mbox{\rm\tiny pl}}}

\newcommand{\ssym}{{\mbox{\rm\tiny sym}}}

\def\sSK{{\tiny \sf SK}}
\def\GOE{{\sf GOE}}
\def\sGOE{{\tiny \sf GOE}}

\def\info{{\sf I}}

\def\bgamma{{\boldsymbol \gamma}}
\def\bG{{\boldsymbol{G}}}
\def\bSigma{{\boldsymbol{\Sigma}}}
\def\bD{{\boldsymbol{D}}}
\def\bu{{\boldsymbol{u}}}
\def\bv{{\boldsymbol{v}}}
\def\bw{{\boldsymbol{w}}}
\def\bx{{\boldsymbol{x}}}

\def\by{{\boldsymbol{y}}}
\def\hby{\hat{\boldsymbol y}}
\def\bz{{\boldsymbol{z}}}
\def\hbz{\hat{\boldsymbol{z}}}
\def\bh{h}

\def\m{{\boldsymbol{m}}}
\def\hm{\hat{\boldsymbol{m}}}
\def\b0{{\boldsymbol{0}}}
\def\bfone{{\boldsymbol{1}}}
\def\normal{{\sf N}}
\def\AMP{{\sf AMP}}
\def\NGD{{\sf NGD}}
\def\tG{\widetilde{G}}

\def\bfe{{\boldsymbol e}}
\def\cE{{\mathcal E}}

\def\sNGD{\mbox{\tiny \sf NGD}}
\def\sAMP{\mbox{\tiny \sf AMP}}
\def\ALG{{\sf ALG}}
\def\sNGD{\mbox{\tiny \sf NGD}}
\def\mmse{{\sf mmse}}
\def\MSE{{\sf MSE}}
\def\Par{{\sf P}}

\def\cuP{\mathscrsfs{P}}
\def\cuF{\mathscrsfs{F}}

\def\cL{{\mathcal L}}

\def\<{{\langle}}
\def\>{{\rangle}}
\def\de{{\rm d}}
\def\cC{{\mathcal C}}
\def\bX{{\boldsymbol{X}}}
\def\bY{{\boldsymbol{Y}}}
\def\sb{{\sf b}}
\DeclareMathOperator*{\plim}{p-lim}

\def\id{{\boldsymbol I}}

\numberwithin{equation}{section}
\numberwithin{theorem}{section}

\begin{document}

\title{Sampling from the Sherrington-Kirkpatrick Gibbs measure via algorithmic stochastic localization}

\author{Ahmed El Alaoui\thanks{Department of Statistics and Data Science, Cornell University}, 
\;\; Andrea Montanari\thanks{Department of Electrical Engineering and Department of Statistics, 
Stanford University}, \;\; Mark Sellke\thanks{Department of Mathematics, Stanford University}}

\date{}
\maketitle

\begin{abstract}
We consider the Sherrington-Kirkpatrick model of spin glasses at high-temperature and
no external field, and study the problem of sampling from the Gibbs distribution 
$\mu$ in polynomial time. We prove that, for any inverse temperature 
$\beta<1/2$, there exists an algorithm
with complexity $O(n^2)$ that samples from a distribution $\mu^{\salg}$
which is close in normalized Wasserstein distance to $\mu$.
Namely, there exists a coupling of $\mu$ and  $\mu^{\salg}$ such that
if $(\bx,\bx^{\salg})\in\{-1,+1\}^n\times \{-1,+1\}^n$ is a pair drawn from this coupling,
then $n^{-1}\E\{\|\bx-\bx^{\salg}\|_2^2\}=o_n(1)$.
The best previous results, by Bauerschmidt and Bodineau \cite{bauerschmidt2019very}
and by Eldan, Koehler, Zeitouni \cite{eldan2021spectral}, implied efficient algorithms to 
approximately sample (under a stronger metric) for $\beta<1/4$.

We complement this result with a negative one, by introducing a suitable 
``stability'' property for sampling algorithms, which is verified by many standard techniques.
We prove that no stable algorithm can approximately sample for $\beta>1$,
even under the normalized Wasserstein metric.

Our sampling method is based on an algorithmic implementation of stochastic localization, which
progressively tilts the measure $\mu$ towards a single configuration, together 
with an approximate message passing algorithm that is used to approximate the mean of 
the tilted measure.
\end{abstract}

\tableofcontents

\section{Introduction}

This Sherrington-Kirkpatrick (SK) Gibbs measure 
is the probability distribution over $\Sigma_n = \{-1,+1\}^n$ given by 
\begin{align}\label{eq:sk}
\mu_{\bA}(\bx) = \frac{1}{Z(\beta,\bA)}\, \exp \Big\{  \frac{\beta}{2}  \langle \bx, \bA \bx \rangle   \Big\} \, ,
\end{align}  
where $\beta \ge  0$ is an inverse temperature parameter and  $\bA \sim \GOE(n)$; i.e., $\bA$ 
is symmetric, $A_{ij} \sim \normal(0,1/n)$ i.i.d.\, for $i \leq j\le n$ and 
$A_{ii} \sim \normal(0,2/n)$, $i\le n$.  The parameter $\beta$ is fixed and we will leave implicit the dependence of $\mu$ upon $\beta$, unless mentioned otherwise.

In this paper, we consider the problem of efficiently sampling from the  
Sherrington-Kirkpatrick spin glass measure. Namely, we seek a 
randomized algorithm that accepts as input $\bA$ and generates $\bx^{\salg}\sim \mu^{\salg}_{\bA}$,
such that: $(i)$~The algorithm runs in polynomial time (for any $\bA$); 
$(ii)$~The distribution  $\mu^{\salg}_{\bA}$ is close to $\mu_{\bA}$ for 
typical realizations of $\bA$. Given  a bounded distance ${\sf dist}(\mu,\nu)$ 
between probability distributions $\mu,\nu$, the second condition can be formalized by 
requiring $\E \{{\sf dist}(\mu_{\bA},\mu^{\salg}_{\bA})\}=o_n(1)$.

Gibbs sampling (also known in this context as Glauber dynamics) provides
an algorithm to approximately sample from $\mu_{\bA}$. However, standard techniques to 
bound its mixing time (e.g., Dobrushin condition \cite{aizenman1987rapid})
only imply polynomial mixing for a vanishing interval of temperatures
$\beta=O(n^{-1/2})$. By contrast, physicists \cite{sompolinsky1981dynamic,SpinGlass} predict fast
convergence to equilibrium (at least for certain observables) for all $\beta<1$.

Significant progress on this question was achieved only recently. In \cite{bauerschmidt2019very},
Bauerschmidt and Bodineau showed that, for $\beta<1/4$, the measure 
$\mu_{\bA}$ can be decomposed into a log-concave mixture of product measures.
They use this decomposition to prove that $\mu_{\bA}$ satisfies a log-Sobolev inequality, although
not for the Dirichlet form of Glauber 
dynamics\footnote{We note in passing that their result immediately suggests a sampling algorithm:
 sample from the log-concave mixture using Langevin dynamics, and then sample from the corresponding
 component using the product form.}.
Eldan, Koehler, Zeitouni \cite{eldan2021spectral} prove that, in the same region $\beta<1/4$,
 $\mu_{\bA}$ satisfies a Poincar\'e inequality for the Dirichlet form of Glauber dynamics. 
 Hence Glauber dynamics mixes in $O(n^2)$ spin flips in total variation distance. This mixing time estimate was improved to $O(n\log n)$ by \cite{anari2021entropic} using a modified log Sobolev inequality, see also \cite[Corollary 51]{chen2022localization}.
 The aforementioned results apply
 deterministically to any matrix $\bA$ satisfying $\beta(\lambda_{\max}(\bA)-\lambda_{\min}(\bA))
 \le 1-\eps$ (for some constant $\eps>0$). 
 
For \emph{spherical} spin glasses, it is shown in \cite{gheissari2019spectral} that Langevin dynamics have a polynomial spectral gap at high temperature. Meanwhile \cite{ben2018spectral} proves that at sufficiently low temperature, the mixing times of Glauber and Langevin dynamics are exponentially large in Ising and spherical spin glasses, respectively.
 
In this paper we develop a different approach which is not based on a 
Monte Carlo Markov Chain strategy. We build on the well known remark that 
approximate sampling can be reduced to approximate computation of expectations of 
the measure $\mu_{\bA}$, and of a family of measures obtained from $\mu_{\bA}$.
One well known method to achieve this reduction is via sequential sampling 
\cite{jerrum1986random,chen2005sequential,blitzstein2011sequential}. 
A sequential sampling approach to $\mu_{\bA}$ would proceed as follows.
Order the variables $x_1,\dots, x_n\in \{-1,+1\}$ arbitrarily. At step $i$ compute
the marginal distribution of $x_i$, conditional to 
$x_1,\dots,x_{i-1}$ taking the previously chosen values:
$p^{(i)}_s := \mu_{\bA}(x_i=s|x_1,\dots,x_{i-1})$, $s\in\{-1,+1\}$. 
Fix $x_i=+1$ with probability $p^{(i)}_{+1}$ and  $x_i=-1$ with probability $p^{(i)}_{-1}$. 

We follow a different route, which is similar in spirit, but that we find
more convenient technically, and of potential  practical interest.
Our approach is motivated by the stochastic localization process \cite{eldan2020taming}. 
Given any probability measure $\mu$ on $\R^n$ with finite second moment, positive time $t>0$,
and vector $\by \in \R^n$, define the tilted measure
\begin{equation}\label{eq:sktilted}
\mu_{\by,t}(\de \bx) :=\frac{1}{Z(\by)} e^{\<\by,\bx\>-\frac{t}{2}\|\bx\|_2^2}\, \mu(\de\bx )\, ,
\end{equation}  
and let its mean vector be
\begin{equation}\label{eq:meanvec}
\m(\by,t) := \int_{\R^n}  \bx \, \mu_{\by,t}(\de\bx)\,.
\end{equation}  
Consider the stochastic differential 
equation\footnote{If $\mu$ is has finite variance, then $\by \to \m(\by,t)$ is Lipschitz 
and so this  SDE is well posed with unique strong solution.} (SDE)
\begin{equation}\label{eq:GeneralSDE}
\rmd \by(t) = \m(\by(t),t) \rmd t + \rmd \bB(t),~~~~ \by(0)=0 \, ,
\end{equation}  
where $(\bB(t))_{t \ge 0}$ is a standard Brownian motion in $\R^n$. 
Then,  the measure-valued process 
$(\mu_{\by(t),t})_{t \ge 0}$ is a martingale
and (almost surely) $\mu_{\by(t),t}\Rightarrow \delta_{\bx^\star}$ as $t\to\infty$,
for some random $\bx^{\star}$
(i.e. the measure localizes). As a consequence of the martingale property, 
$\E[\int \varphi(\bx)\mu_{\by(t),t}(\de\bx)]$  is a constant for any 
bounded continuous function $\varphi$, whence
$\E[\varphi(\bx^{\star})] = \int \varphi(\bx)\mu(\de\bx)$. In other words, $\bx^{\star}$
is a sample from $\mu$.
 For further information on this process, we refer to Section~\ref{sec:stochloc}.

In order to use this process as an algorithm to sample from the SK measure
$\mu=\mu_{\bA}$, we need to overcome two problems:
\begin{itemize}
\item \emph{Discretization.} We need to discretize the SDE \eqref{eq:GeneralSDE} in time,
and still guarantee that the discretization closely tracks the original process.  
This is of course possible only if the map $\by\mapsto \m(\by,t)$ is sufficiently regular.
\item \emph{Mean computation.} We need to be able to compute the mean vector 
$\m(\by,t)$ efficiently. To this end, we use an approximate message passing (AMP) algorithm
for which we can leverage earlier work \cite{deshpande2017asymptotic} to establish that
$\|\m(\by)-\hm_{\AMP}(\by)\|^2_2/n=o_n(1)$ along the algorithm trajectory.
(Note that the SK measure is supported on vectors with $\|\bx\|_2^2=n$, and
hence the quadratic component of the tilt in Eq.~\eqref{eq:sktilted} drops out.
We will therefore write $\m(\by)$ or $\m(\bA,\by)$ instead $\m(\by,t)$ for the mean of the Gibbs measure.)
\end{itemize}
To our knowledge, ours is the first algorithmic implementation of the 
stochastic localization process, although a recent paper by Nam, Sly and Zhang~\cite{nam2022ising} uses this process (without naming it as such) to show that the Ising measure on the infinite regular tree is a factor of IID process up to a constant factor away from the Kesten--Stigum, or ``reconstruction'', threshold. Their construction can easily be transformed into a sampling algorithm.

In order to state our results, we define the normalized 2-Wasserstein distance between two probability 
measures $\mu , \nu$ on $\R^n$ with finite second moments as
\begin{equation}  
W_{2,n}(\mu,\nu)^2 =  \inf_{\pi \in \cC(\mu,\nu)} \frac{1}{n} 
\E_{\pi} \Big[\big\|\bX - \bY\big\|_2^2\Big] \, , 
\end{equation}   
where the infimum is over all couplings $(\bX,\bY) \sim \pi$ with marginals $\bX \sim \mu$ and 
$\bY \sim \nu$.  

In this paper, we establish two main results.
\begin{description}
\item[Sampling algorithm for $\beta<1/2$.] We prove that the strategy outlined
above yields an algorithm with complexity $O(n^2)$, which samples
from a distribution $\mu_{\bA}^{\salg}$ such that 
$W_{2,n}(\mu_{\bA}^{\salg},\mu_{\bA}) = o_{\mathbb P,n}(1)$.
\item[Hardness for stable algorithms, for $\beta>1$.] We prove that no algorithm satisfying a certain \emph{stability}
property can sample from the SK measure (under the
same criterion $W_{2,n}(\mu_{\bA}^{\salg},\mu_{\bA}) = o_{\mathbb P,n}(1)$) for 
$\beta>1$, i.e., when replica symmetry is broken. Roughly speaking, stability formalizes the notion that the algorithm output 
behaves continuously with respect to the matrix $\bA$. 
\end{description}
It is worth pointing out that we expect our algorithm to be successful (in the sense described above)
for all $\beta<1$ and that closing the gap between $\beta=1/2$ and $\beta=1$
should be within reach of existing techniques, at the price of a longer technical argument. We expound on this point in Remark~\ref{rmk:Beta} further below, and in Section~\ref{sec:ngd}.  

The hardness results for $\beta>1$ are proven using the notion of 
disorder chaos, in a similar spirit to the use of the \emph{overlap gap property} for
random optimization, estimation, and constraint satisfaction 
problems~\cite{gamarnik2014limits, rahman2017independent, gamarnik2017performance, 
chen2019suboptimality, gamarnik2019overlap, gamarnik2020optimization, wein2020independent,
 gamarnik2021partitioning, bresler2021ksat, gamarnik2021circuit, huang2021tight}. 
 While the overlap gap property has been used to rule out stable algorithms for this class of 
 problems, and variants have been used to rule out efficient sampling by specific Markov chain 
 algorithms, to the best of our knowledge we are the first to rule out stable sampling algorithms 
 using these ideas. In sampling there is no hidden solution or set of solutions to be found, 
 and therefore no notion of an overlap gap in the most natural sense. Instead, we argue directly that the distribution to be sampled from is unstable in a $W_{2,n}$ sense at low temperature, and hence cannot be approximated by any stable algorithm.

The rest of the paper is organized as follows. In Section \ref{sec:Main}
we formally state our results. In Section \ref{sec:stochloc} we collect some useful properties of the
stochastic localization process, and we present the analysis of our algorithm in 
Section \ref{sec:analysis}. Finally, the proof of hardness under stability 
is given in Section \ref{sec:stable}.

\section{Main results}
\label{sec:Main}

\subsection{Sampling algorithm for $\beta<1/2$}

In this section we describe the sampling algorithm, and formally
state the result of our analysis. As pointed out in the introduction, 
a main component is the computation of the mean of the tilted SK measure:
\begin{align}\label{eq:mu_tilted}
\mu_{\bA,\by}(\bx):= \frac{1}{Z(\bA,\by)} \exp\Big\{\frac{\beta}{2}
\<\bx,\bA\bx\>+ \<\by,\bx\>\Big\}\, ,\;\;\;\; \bx\in\{-1,+1\}^n\, .
\end{align}
We describe the algorithm to approximate this mean in Section 
\ref{sec:MeanApprox}, the overall sampling procedure 
(which uses this estimator as a subroutine) in Section \ref{sec:SamplingAlg},
and our Wasserstein-distance guarantee in  Section \ref{sec:SamplingTheorem}.

\subsubsection{Approximating the mean of the Gibbs measure}
\label{sec:MeanApprox}

\begin{algorithm}
\label{alg:Mean}
\DontPrintSemicolon % Some LaTeX compilers require you to use \dontprintsemicolon instead
\KwIn{Data $\bA\in\R^{n\times n}$, $\by\in \mathbb R^n$, parameters $\beta,\eta>0$, $q\in (0,1)$, iteration numbers 
$K_{\sAMP}$, $K_{\sNGD}$.}
$\hm^{-1} = \bz^{0}= 0$,\\
\For{$k = 0,\cdots,K_{\sAMP}-1$} { 
$\hm^{k} = \tanh(\bz^{k} ) , ~~~~~~~ \sb_{k}= \frac{\beta^2}{n}\sum_{i=1}^n 
\big(1-\tanh^2(z^{k}_i) \big)$\, ,\\
$\bz^{k+1} = \beta \bA \hm^{k} + \by - \sb_{k} \hm^{k-1}$\, , \label{eq:AMP-main-step}\\
}
$\bu^0 = \bz^{K_{\sAMP}}$, \label{alg:NGD-begin}\\
\For{$k = 0,\cdots,K_{\sNGD}-1$} { 
$\bu^{k+1} = \bu^k - \eta \cdot\nabla \cuF_{\sTAP}(\hm^{+,k};\by,q)$,   \\
$\hm^{+,k+1} = \tanh(\bu^{k+1})$,
}
\Return{$\hm^{+,K_{\sNGD}}$}\label{alg:NGD-end}\\
\caption{{\sc Mean of the tilted Gibbs measure}}
\end{algorithm}

We will denote our approximation of the mean of the Gibbs measure 
$\mu_{\bA,\by}$ by $\hm(\bA,\by)$, while the actual mean will be $\m(\bA,\by)$.

 The algorithm to compute
$\hm(\bA,\by)$ is given in Algorithm \ref{alg:Mean}, and is composed of two phases: 
\begin{enumerate}
\item An Approximate Message Passing (AMP) 
algorithm is run for $K_{\sAMP}$ iterations and constructs a first estimate  of
the mean. 
We denote by $\AMP(\bA,\by ; k)$ the estimate produced after $k$
AMP iterations
\begin{equation} 
\label{eq:AMP}
\AMP(\bA,\by ; k) := \hm^{k}\,.
\end{equation}
\item Natural gradient descent (NGD) is run for $K_{\sNGD}$ iterations
with initialization given by vector computed at the end of the first phase.
This phase attempts to minimize the following version of the TAP free energy (for a specific value of $q$):
\begin{align}
\cuF_{\sTAP}(\m ; \by, q) &:= -\frac{\beta}{2}  \langle \m, \bA \m \rangle - \langle \by , \m \rangle - \sum_{i=1}^n h(m_i) - \frac{n \beta^2 (1-q)(1+q - 2Q(\m))}{4} \, , \\
Q(\m) &= \frac{1}{n} \|\m\|^2 ,~~~~~~ h(m) = -\frac{1+m}{2}\log \left(\frac{1+m}{2}\right) - \frac{1-m}{2}\log \left(\frac{1-m}{2}\right)  \, .
\end{align}
\end{enumerate}
The second stage is motivated by the TAP (Thouless-Anderson-Palmer) equations for the Gibbs mean of a high-temperature spin glass \cite{SpinGlass,TalagrandVolI}. Essentially by construction, stationary points for the function $\cuF_{\sTAP}(\m ; \by, q)$ satisfy the TAP equations, and we show in Lemma~\ref{lem:TAP-stationary} that the first stage above constructs an approximate stationary point for $\cuF_{\sTAP}(\m ; \by, q)$. The effect of the second stage is therefore numerically small, but it turns out to reduce the error incurred by discretizing time in line~\ref{step:DiscreteSDE} of Algorithm~\ref{alg:Sampling}.

Let us emphasize that this two-stage construction is considered for technical reasons. 
Indeed a simpler algorithm, that runs AMP for a larger number of 
iteration, and does not run NGD at all, is expected to work but our arguments do not go through. 
The hybrid algorithm above allows us to exploit known properties of AMP (precise analysis via state evolution)
and of $\cuF_{\sTAP}(\m ; \by, q)$ (Lipschitz continuity of the minimizer in $\by$).

\subsubsection{Sampling via stochastic localization}
\label{sec:SamplingAlg}

\begin{algorithm}
\label{alg:Sampling}
\DontPrintSemicolon % Some LaTeX compilers require you to use \dontprintsemicolon instead
\KwIn{Data $\bA\in\R^{n\times n}$, parameters $(\beta,\eta,K_{\sAMP},K_{\sNGD},L,\delta)$}
$\hby_0=0$,\\
\For{$\ell = 0,\cdots,L-1$} {
Draw $\bw_{\ell+1}\sim\normal(0,\bI_n)$ independent of everything so far;\\
Set $q= q_{*}(\beta,t=\ell\delta)$;\\
Set $\hm(\bA,\hby_{\ell})$ the output of Algorithm \ref{alg:Mean}, 
with parameters $(\beta,\eta,q,K_{\sAMP},K_{\sNGD})$;\\
Update $\hby_{\ell+1} = \hby_{\ell} + \hm(\bA,\hby_{\ell}) \, \delta + \sqrt{\delta} \, \bw_{\ell+1}$
\label{step:DiscreteSDE}
}
Set $\hm(\bA,\hby_{L})$ the output of Algorithm \ref{alg:Mean}, 
with parameters $(\eta,q,K_{\sAMP},K_{\sNGD})$;\\
Draw $\{x_i^{\salg}\}_{i\le n}$  conditionally independent
with 
$\E[x_i^{\salg}|\by,\{\bw_{\ell}\}] = \widehat{m}_i(\bA,\hby_{L})$\\
\Return{$\bx^{\salg}$}\;
\caption{{\sc Approximate sampling from the SK Gibbs measure}}
\end{algorithm}

Our sampling algorithm is presented as Algorithm \ref{alg:Sampling}.
The algorithm makes uses of constants $q_k:=q_k(\beta,t)$. With $W\sim \normal(0,1)$ a standard Gaussian, these constants are defined 
for $k,\beta,t\ge 0$ by the recursion
\begin{align}
q_{k+1}& = \E\Big\{\tanh\big(\beta^2q_k+t+\sqrt{\beta^2q_k+t}\,W\big)^2\Big\}\, ,\;\;\; q_0=0\, ,
\;\;\;
q_* = \lim_{k\to\infty}q_k\, .
\end{align}
This iteration can be implemented via a one-dimensional integral,
and the limit $q_*$ is approached exponentially fast in $k$ (see Lemma \ref{lem:properties} below). 
The values
$q_*(\beta,t=\ell\delta)$ for $\ell\in\{0,\dots,L\}$ can be precomputed
and are independent of the input $\bA$. For the sake of simplicity, we will neglect errors
in this calculation. 

The core of the sampling procedure is step \ref{step:DiscreteSDE},
which is a standard Euler discretization of the SDE \eqref{eq:GeneralSDE},
with step size $\delta$, over the time interval $[0,T]$, $T=L\delta$. The
mean of the Gibbs measure $\m(\bA,\by)$ is replaced by the output 
of Algorithm \ref{alg:Mean} which we recall is denoted by $\hm(\bA,\by)$.
We reproduce the Euler iteration here for future reference
\begin{equation}\label{eq:approx}
\hby_{\ell+1} = \hby_{\ell} + \hm(\bA,\hby_{\ell}) \, \delta + \sqrt{\delta} \, \bw_{\ell+1}\, .
\end{equation}

The output of the iteration is $\hm(\bA,\hby_{L})$, which should be thought of as an approximation of 
$\m(\bA,\by(T))$, $T=L\delta$, that is the mean of $\mu_{\bA,\by(T)}$. 
According to the discussion in the introduction,
for large $T$, $\mu_{\bA,\by(T)}$ concentrates around $\bx^{\star}\sim \mu_{\bA}$. In other words,
$\m(\bA,\by(T))$ is close to the corner $\bx^{\star}$ of the hypercube. We
round its coordinates independently to produce the output $\bx^{\salg}$.

\subsubsection{Theoretical guarantee}
\label{sec:SamplingTheorem}

Our main positive result is the following.
\begin{theorem}
\label{thm:main}
For any $\eps>0$ and $\beta_0< 1/2$ there exist 
$\eta,K_{\sAMP},K_{\sNGD},L,\delta$ independent of $n$, so that the following
holds for all $\beta\le \beta_0$.
The sampling algorithm~\ref{alg:Sampling} takes as input $\bA$  and
parameters  $(\eta,K_{\sAMP},K_{\sNGD},L,\delta)$
and outputs a random point
 $\bx^{\salg} \in \{-1,+1\}^n$ with law $\mu_{\bA}^{\salg}$ such that with probability $1-o_n(1)$ 
 over $\bA\sim\GOE(n)$,
\begin{equation}
\label{eq:main}
    W_{2,n}( \mu_{\bA}^{\salg} , \mu_{\bA}^{\phantom{\salg}} ) \leq \eps \, .
\end{equation}    
The total complexity of this algorithm is $O(n^2)$.
\end{theorem}
\begin{remark}\label{rmk:Beta}
The condition $\beta<1/2$ arises because our proof requires
the Hessian of the TAP free energy to be positive definite at its minimizer.
A simple calculation yields
\begin{align}
\nabla^2\cuF_{\sTAP}(\m;\by,q) = -\beta\bA +\bD(\m)
+\beta^2(1-q)\, \id_n\, ,\;\;\;\;\;\;\, \bD(\m):={\rm diag}\big(\{(1-m_i^2)^{-1}\}_{i\le n}\big).
\end{align}
A crude bound yields $\nabla^2\cuF_{\sTAP}(\m;\by,q) \succeq -\beta\bA+\id_n
\succeq (1-\beta\lambda_{\max}(\bA))\id_n$. Since
$\plim_{n\to\infty}\lambda_{\max}(\bA)= 2$ the desired condition holds trivially for $\beta<1/2$. 
However, we expect that a more careful treatment will reveal that the Hessian is locally positive in a neighborhood of the minimizer for all $\beta<1$.
\end{remark}

\subsection{Hardness for stable algorithms, for $\beta>1$}

The sampling algorithm \ref{alg:Sampling} enjoys stability properties
with respect to changes in the inverse temperature $\beta$ and the matrix $\bA$
which are shared by many natural efficient algorithms. We will 
use the fact that the actual Gibbs measure does not enjoy this stability property
for $\beta>1$ to conclude that sampling is hard for
all stable algorithms.

Throughout this section, we denote the
Gibbs and algorithmic output distributions by $\mu_{\bA,\beta}$ and $\mu_{\bA,\beta}^{\salg}$ 
respectively to emphasize the dependence on $\beta$.
\begin{definition}\label{def:Stable}
Let $\{\ALG_n\}_{n\geq 1}$ be a family of randomized sampling algorithms, i.e., measurable 
maps  
\[
    \ALG_n:(\bA,\beta,\omega) \mapsto \ALG_n(\bA,\beta,\omega)\in [-1,1]^n \, ,
\]
where $\omega$ is a random seed (a point in a probability space $(\Omega,\cF,\P)$). 
Let $\bA'$ and $\bA\sim \GOE(n)$ be independent copies of the coupling matrix, and
consider perturbations $\bA_s=\sqrt{1-s^2} \bA + s \bA'$ for $s \in [0,1]$.
Finally, denote by $\mu_{\bA_s,\beta}^{\salg}$ the law of the algorithm output,
i.e., the distribution of $\ALG_n(\bA_s,\beta,\omega)$ when $\omega\sim \P$ independent 
of $\bA_s,\beta$ which are fixed.

We say $\ALG_n$ is \emph{stable with respect to disorder}, at inverse temperature $\beta$, if
\begin{align}
\lim_{s\to 0} \, \plim_{n\to\infty} \, W_{2,n}(\mu_{\bA,\beta}^{\salg},\mu_{\bA_s,\beta}^{\salg})=0\, .
\end{align}

We say $\ALG_n$ is \emph{stable with respect to temperature} at inverse temperature $\beta$, if
\begin{align}
\lim_{\beta'\to\beta}\, \plim_{n\to\infty} \, W_{2,n}(\mu_{\bA,\beta}^{\salg},\mu_{\bA,\beta'}^{\salg})=0\, .
\end{align}
\end{definition}

We begin by establishing the stability of the proposed sampling algorithm. 
\begin{theorem}[Stability of the sampling Algorithm \ref{alg:Sampling}]
\label{thm:stable}
For any $\beta\in (0,\infty)$ and fixed parameters 
$(\eta,$ $K_{\sAMP},$ $K_{\sNGD},$ $L,$ $\delta)$, Algorithm \ref{alg:Sampling} is
stable with respect to disorder and with respect to temperature. 
\end{theorem}

This theorem is proved in Section~\ref{sec:algo_stable}.
As a consequence, the Gibbs measures $\mu_{\bA,\beta}$ enjoy similar stability 
properties for $\beta<1/2$, which amount (as discussed below) to the absence of chaos in 
both temperature and disorder:
\begin{corollary}
\label{cor:stable}
For any $\beta<1/2$, the following properties hold for the Gibbs measure 
$\mu_{\bA,\beta}$ of the Sherrington-Kirkpatrick model, cf. Eq.~\eqref{eq:sk}:
\begin{enumerate}
    \item \label{it:disorder-stability} $\lim_{s\to 0}\plim_{n\to\infty} W_{2,n}(\mu_{\bA,\beta},\mu_{\bA_s,\beta})=0$.
    \item $\lim_{\beta'\to\beta}\plim_{n\to\infty} W_{2,n}(\mu_{\bA,\beta},\mu_{\bA,\beta'})=0$.
\end{enumerate}
\end{corollary}

\begin{proof}
Take $\eps>0$ arbitrarily small and choose  parameters 
$(\eta,K_{\sAMP},K_{\sNGD},L,\delta)$ of Algorithm \ref{alg:Sampling}
with the desired tolerance $\eps$ so that Theorem~\ref{thm:main} holds. 
Combining with Theorem~\ref{thm:stable} using the same parameters $(\eta,K_{\sAMP},K_{\sNGD},L,\delta)$
implies the result since $\eps$ is arbitrarily small. 
(Recall that $(\eta,K_{\sAMP},K_{\sNGD},L,\delta)$ can be chosen independent of $\beta$
for $\beta\leq \beta_0<1/2$.)
\end{proof}

\begin{remark}
 We emphasize that Corollary \ref{cor:stable} makes no 
reference to the sampling algorithm, and is instead a purely structural property of
the Gibbs measure. The sampling algorithm, however, is the key tool of our proof.
\end{remark}

Stability is related to chaos, which is a well studied and 
important property of spin glasses, see e.g.
\cite{chatterjee2009disorder,chen2013disorder,chatterjee2014superconcentration,chen2015disorder,
chen2018disorder}. In particular, ``disorder chaos'' refers to 
the following phenomenon. Draw $\bx^0\sim \mu_{\bA,\beta}$ independently of  
$\bx^s\sim \mu_{\bA_s,\beta}$, and denote by $\mu^{(0,s)}_{\bA,\beta}:=\mu_{\bA,\beta}\otimes
\mu_{\bA^s,\beta}$ their joint distribution. Disorder chaos holds at inverse temperature
$\beta$ if 
\begin{align}
\lim_{s\to 0}\lim_{n\to\infty}\E\mu^{(0,s)}_{\bA,\beta} \Big\{\Big(\frac{1}{n}\<\bx^0,\bx^s\>\Big)^2\Big\}= 0
\, .\label{eq:FirstDisorderChaos}
\end{align}
Note that disorder chaos is not necessarily a surprising property.
For instance when $\beta=0$, the distribution $\mu_{\bA_s,\beta}$ is simply the uniform measure over the hypercube $\{-1,+1\}^n$ for all $s$, and this example exhibits disorder chaos
in the sense of Eq.~\eqref{eq:FirstDisorderChaos}. 
In fact, the SK Gibbs measure exhibits disorder chaos at all $\beta\in [0,\infty)$ 
\cite{chatterjee2009disorder}. However, for $\beta>1$, 
Eq.~\eqref{eq:FirstDisorderChaos} leads to a stronger conclusion.
\begin{theorem}[Disorder chaos in $W_{2,n}$ distance] \label{thm:disorder_chaos_sk}
For all $\beta>1$, 
\[
   \inf_{s\in (0,1)} \, \liminf_{n\to\infty} \, \E\big[W_{2,n}(\mu_{\bA,\beta},\mu_{\bA_s,\beta})\big]>0 \, .
\]
%(In the above, the expectation is with respect to the matrix $\bA$ and $\beta$.)
\end{theorem}

 Finally, we obtain the desired hardness result by reversing the implication 
 in Corollary \ref{cor:stable}:
 no stable algorithm which can approximately sample from the measure 
 $\mu_{\bA,\beta}$ in the $W_{2,n}$ sense for $\beta>1$.
\begin{theorem}  
\label{thm:disorder-stable-LB}
Fix $\beta>1$, and let $\{\ALG_n\}_{n\geq 1}$ be a family of randomized algorithms  
which is stable with respect to disorder as per Definition \ref{def:Stable}
at inverse temperature $\beta$. 
Let $\mu^{\salg}_{\bA,\beta}$ be the law of the output 
$\ALG_n(\bA,\beta,\omega)$ conditional on $\bA$. Then
\[
   \liminf_{n\to\infty} \E \big[W_{2,n}(\mu^{\salg}_{\bA,\beta},~\mu_{\bA,\beta}) \big] >0 \, .
\]
\end{theorem}
 We refer the  reader to Section~\ref{sec:disorder} for the proof of this theorem.

\subsection{Notations}

We use $o_n(1)$ to indicate a quantity tending to $0$ as $n\to\infty$. We use $o_{n,\P}(1)$ 
for a quantity tending to $0$ in probability. If $X$ is a random variable, 
then $\mathcal L(X)$ indicates its law. The quantity $C(\beta)$ refers to a constant depending 
on $\beta$. For $\bx\in\R^n$ and $\rho\in\R_{\geq 0}$, we denote the open ball 
of center $\bx$ and radius $\rho$ by 
$B(\bx,\rho):=\{\by\in \R^n: \|\by-\bx\|_2<\rho\}$. The uniform distribution on the interval 
$[a,b]$ is denoted by $\Unif([a,b])$.
The set of probability distributions over a measurable space $(\Omega,\cF)$
is denoted by $\cuP(\Omega)$.

\section{Properties of stochastic localization}
\label{sec:stochloc}

We collect in this section the main properties of the stochastic localization 
process needed for our analysis. To be definite, we will focus
on the stochastic localization process for the Gibbs measure \eqref{eq:sk},
although most of what we will say generalizes to other probability measures in $\R^n$,
 under suitable tail conditions. Throughout this section, the matrix
 $\bA$ is viewed as fixed.
 
Recalling the tilted measure $\mu_{\bA,\by}$ of Eq.~\eqref{eq:sktilted},
and the SDE of Eq.~\eqref{eq:GeneralSDE}, we introduce the shorthand
\[
 \mu_t = \mu_{\bA,\by(t)}\, .
 \]

 The following properties are well known. See for instance~\cite[Propositions 9, 10]{eldan2022log} 
 or~\cite{eldan2020taming}. We provide proofs for the reader's convenience.
 \begin{lemma}\label{prop:stochloc1}
For all $t \ge 0$ and all $\bx \in \{-1,+1\}^n$, 
\begin{equation}\label{eq:L}
 \rmd \mu_t(\bx) = \mu_t(\bx) \langle \bx - \m_{\bA,\by(t)} , \rmd \bB(t)\rangle  \, .
 \end{equation}
As a consequence, for any function $\varphi : \R^n \to \R^m$, the process 
$\big(\E_{\bx \sim \mu_t}\big[\varphi(\bx)\big]\big)_{t \ge 0}$ is a martingale.  
 \end{lemma}
 \begin{proof}
 Let us evaluate the differential of $\log \mu_t$. By writing $Z_t$
 for the normalization constant $Z(\by(t))$ of Eq.~\eqref{eq:sktilted}, we get
\begin{equation}\label{eq:logL}
\rmd \log \mu_t (\bx) =  \langle \rmd \by(t) , \bx\rangle - \rmd \log Z_t\, .
\end{equation}
 Using It\^{o}'s formula for $Z_t$ we have
\begin{align*}
\rmd Z_t &= \rmd \sum_{\bx \in \{-1,+1\}^n} e^{(\beta/2)  \langle \bx, \bA \bx \rangle  + \langle \by(t) ,\bx\rangle}\\
&=  \sum_{\bx \in \{-1,+1\}^n}  \big(\langle \rmd  \by(t) , \bx\rangle + \frac{1}{2}\|\bx\|^2 \rmd t\big) e^{(\beta/2)  \langle \bx, \bA \bx \rangle  + \langle  \by(t) , \bx\rangle} \, .
\end{align*}
Therefore, denoting by $[Z]_t$ the quadratic variation process associated to $Z_t$,
\begin{align*}
\rmd \log Z_t &= \frac{\rmd Z_t}{Z_t} - \frac{1}{2} \frac{\rmd [Z]_t}{Z_t^2} \\
&= \langle \rmd  \by(t) , \m_{\bA,\by(t)}\rangle  +  \frac{1}{2} \E_{\mu_t}[\|\bx\|^2] \rmd t -  \frac{1}{2} \|\m_{\bA,\by(t)}\|^ 2 \rmd t \, \\
&= \langle \rmd  \by(t) , \m_{\bA,\by(t)}\rangle  +  \frac{n}{2}  -  \frac{1}{2} \|\m_{\bA,\by(t)}\|^ 2 \rmd t \, .
\end{align*}
Substituting in~\eqref{eq:logL} we obtain
\begin{align*}
\rmd \log \mu_t (\bx) &= \langle \rmd  \by(t) , \bx - \m_{\bA,\by(t)}\rangle -  \frac{n}{2} \rmd t +  \frac{1}{2} \|\m_{\bA,\by(t)}\|^ 2 \rmd t \\
&= \langle \rmd \bB_t , \bx - \m_{\bA,\by(t)}\rangle - \frac{1}{2} \|\bx - \m_{\bA,\by(t)}\|^ 2 \rmd t \, .
\end{align*}
Applying It\^{o}'s formula to $e^{\log \mu_t(\bx)}$ yields the desired result. 

Finally, Eq.~\eqref{eq:L} implies that $\mu_t(\bx)$ is a martingale for every
$\bx\in\{-1,+1\}^n$. Since $\E_{\bx \sim \mu_t}\big[\varphi(\bx)\big]$ is
a linear combination of martingales, it is itself a martingale.
\end{proof}

 \begin{lemma}[\cite{eldan2020taming}]\label{lem:covbound}
For all $t >0$,
\begin{equation}\label{eq:covbound}
\E \cov(\mu_t) \preceq \frac{1}{ t } \bI_n \, .
 \end{equation}
 \end{lemma}

\begin{lemma}\label{lem:W2bound}
For all $t>0$, 
\begin{equation}\label{eq:W2bound}
W_{2,n}\big(\mu_{\bA}, \cL(\m_{\bA,\by(t)})\big)^2 \le \frac{1}{t}\, .
 \end{equation}
In particular, the mean vector $\m_{\bA,\by(t)}$ converges in distribution to a random vector $\bx^\star \sim \mu_{\bA}$ as $t\to \infty$. 
\end{lemma}
\begin{proof}
 By Lemma~\ref{lem:covbound},
 \[\E \big[\E_{\bx \sim \mu_t} [\|\bx - \m_{\bA,\by(t)}\|^2]  \big]\le \frac{n}{t} \, ,\]  
 therefore 
\[\E \Big[W_{2,n}\big(\mu_t, \delta_{\m_{\bA,\by(t)}}\big)^2\Big] \le \frac{1}{t} \, .\] 
Notice that $(\mu,\nu)\mapsto W_{2,n}^2(\mu,\nu)$ is jointly convex.
Since $\mu_{\bA} = \E [\mu_t] $, this implies
\[
    W_{2,n}\big(\mu_{\bA}, \cL(\m_{\bA,\by(t)})\big)^2 \le \E \Big[W_{2,n}\big(\mu_t, \delta_{\m_{\bA,\by(t)}}\big)^2\Big] \le \frac{1}{t}\, .
    \qedhere
\]
\end{proof}

%%%%%%%%%%%%%%%%%

\section{Analysis of the sampling algorithm and proof of Theorem \ref{thm:main}}
\label{sec:analysis}

This section is devoted to the analysis of Algorithm~\ref{alg:Sampling} described in the previous section.
An important simplification is obtained by reducing ourselves to 
working with a corresponding \emph{planted} model. This approach has two advantages:
$(i)$~The joint distribution of the matrix $\bA$ and the process $(\by(t))_{t\ge 0}$ 
in~\eqref{eq:GeneralSDE} is significantly simpler in the planted model;
$(ii)$~Analysis in the planted model can be cast as a statistical estimation problem.
In the latter, Bayes-optimality considerations can be exploited to relate the output of 
the AMP algorithm $\AMP(\bA,\by;k)$ to the true mean vector 
$\m(\bA,\by)$.       

This section is organized as follows. Section \ref{sec:planted} introduces the 
planted model and its relation to the original model. We then analyze the AMP
component of our algorithm in Section \ref{sec:AMP}, and the NGD component in
Section \ref{sec:ngd}. Finally, Section \ref{sec:proofMain} puts the various elements
together and proves Theorem \ref{thm:main}.

%%%%%%%%%%%%%%%%
\subsection{The planted model and contiguity}
\label{sec:planted}
Let $\onu$ be the uniform distribution over $\{-1,+1\}^n$ and consider the joint distribution of pairs $(\bx,\bA) \in   \{-1,+1\}^n \times \R^{n \times n}_{\ssym} $,
\begin{equation}
 \mu_{\pl}(\rmd \bx, \rmd \bA) = \frac{1}{Z_{\pl}}\, \exp\Big\{-\frac{n}{4} \Big\|\bA - \frac{\beta \bx\bx^\top}{n} \Big\|_{F}^2 \, \Big\}  \, \onu(\rmd \bx) \, \rmd \bA \, ,
\end{equation}
where $\rmd \bA$ is the Lebesgue measure over  the space of symmetric matrices $\R^{n \times n}_{\ssym} $, and the normalizing constant
\begin{align}\label{eq:Zpl}
    Z_{\pl} := \int\! \exp\Big\{-\frac{n}{4}\Big\|\bA-\frac{\beta \bx\bx^{\top}}{n}\Big\|_F^2\Big\}\, \rmd \bA 
\end{align}
is independent of $\bx \in \{-1,+1\}^n$. 
It is easy to see by construction that the marginal distribution of $\bx$ under 
$\mu_{\pl}$ is $\onu$, and the conditional law $\mu_{\pl}( \, \cdot \, | \bx)$ is a rank-one spiked GOE model
 with spike $\beta \bx \bx^\top/n$. Namely, under $\mu_{\pl}( \, \cdot \, | \bx)$,
 we have
 \begin{align}
 \bA = \frac{\beta}{n} \bx \bx^\top+\bW\, ,\;
 \;\; \bW\sim\GOE(n)\, .
 \end{align}
  On the other hand, $\mu_{\pl}( \, \cdot \, | \bA)$ is the SK measure $\mu_{\bA}$.  

% Now we consider  of the sampling procedure: instead of first sampling $\bA$ from 
% $\mu^{\sGOE} = \GOE(n)$ then $\bx \sim \mu_{\bA} = \mu( \, \cdot \, | \bA)$, we sample 
% $\bx$ from its marginal $\onu$ and then take $\bA \sim \mu(\, \cdot \, | \bx)$. In doing so,
 %observe however that t
  The marginal of $\bA$ under $\mu_{\pl}$ is not the $\GOE(n)$ 
  distribution $\mu_{\sGOE}$ but takes the form
 \begin{align}
 \mu_{\pl}(\rmd \bA) &= \frac{1}{Z_{\pl}} \, e^{-\frac{n}{4}\|\bA\|_{F}^2} \, Z_{\sSK}(\bA)\, \rmd \bA \\
 &= \mu_{\sGOE}(\rmd \bA)\, Z_{\sSK}(\bA)  \, ,
 \end{align}
where $Z_{\sSK}(\bA)$ is the (rescaled) partition function of the SK measure
 \begin{equation}
Z_{\sSK}(\bA) = \, 2^{-n} \hspace{-.3cm} \sum_{\bx \in \{-1,+1\}^n} \exp\Big\{\frac{\beta}{2}  \langle \bx, \bA \bx \rangle - \frac{\beta^2 n}{4}\Big\} \, .
 \end{equation}
By a classical result of~\cite{aizenman1987some}\footnote{As stated in \cite{aizenman1987some}, the variance of $W$ is different because an SK model with zero diagonal entries is considered, which suggests $\sigma^2=\frac{1}{4} (-\log(1-\beta^2) - \beta^2)$. It is easy to see that including the diagonal entries increases the variance up to the stated value of $\sigma^2$ in the statement of Theorem~\ref{thm:ALR}, since these diagonal entries affect all points in $\{-1,+1\}^n$ equally. In any case, the numerical value of $\sigma^2$ is immaterial for our purposes.}, $Z_{\sSK}(\bA)$ has log-normal fluctuations for all $\beta<1$:
\begin{theorem}[\cite{aizenman1987some}]\label{thm:ALR}
Let $\beta<1$, $\bA \sim \mu_{\sGOE}$ and $\sigma^2 = \frac{-\log(1-\beta^2)}{4}$. Then
\begin{equation}
Z_{\sSK}(\bA) \xrightarrow[n \to \infty]{\rmd} \exp (W) \, ,
\end{equation}
where $W\sim \normal \big(- \sigma^2, 2\sigma^2\big)$.
\end{theorem}
 Therefore, by Le Cam's first lemma~\cite[Lemma 6.4]{van1998asymptotic}, $ \mu_{\pl}(\de\bA)$ and 
 $ \mu_{\sGOE}(\de\bA)$ are mutually contiguous for all $\beta<1$.     
For the purpose of our analysis we will need a stronger result about the joint distributions of 
$(\bA, \by)$ under our ``random'' model and a planted model which we now introduce.  

Recall that $\m(\bA,\by)$ denotes the mean of the Gibbs measure $\mu_{\bA,\by}$ in
Eq.~\eqref{eq:sktilted}.
For a fixed $T \ge 0$, we define two Borel distributions $\P$ and $\Q$ on 
$(\bA,\by)\in \R^{n \times n}_{\ssym} \times C([0,T], \R^n)$ as follows:
\begin{align}   
\Q ~~&:~~
\begin{dcases}
\bA &\sim~~ \mu_{\sGOE} \, , \\
\by(t) &=~~ \int_{0}^t \m(\bA,\by(s)) \, \rmd s + \bB(t) \, , ~~~ t \in [0,T] \, ,
\end{dcases}
\hspace{.5cm} &\text{(random)}\label{eq:Q}
\\
~
\P ~~&:~~
\begin{dcases}
\bx_0 &\sim~~ \onu \, , \\
\bA &\sim~~ \mu_{\pl}( \,\cdot\, |\, \bx_0) \, ,\\%~~\frac{\beta}{\sqrt{n}} \m \m^\top + \bW \, \\
\by(t) &=~~ t \bx_0 + \bB(t) \, , ~~~ t \in [0,T] \, \,
\end{dcases}
\hspace{.5cm} &\text{(planted)}\label{eq:P}
\end{align}
where $(\bB(t))_{t \ge 0}$ is a standard Brownian motion in $\R^n$ independent of everything else.
Note the SDE defining the process $\by = (\by(t))_{t\in [0,T]}$ in 
Eq.~\eqref{eq:Q} is a restatement of the stochastic localization equation~\eqref{eq:GeneralSDE} 
applied to the SK measure $\mu_{\bA}$. 
\begin{proposition}\label{prop:contig}
For all $T \ge 0$ and $\beta \ge 0$,  $\P$ absolutely continuous with respect to $\Q$ and for all $(\bA,\by) \in  \R^{n \times n}_{\ssym}  \times  C([0,T], \R^n)$, 
\[ \frac{\rmd \P}{\rmd \Q}(\bA, \by) = Z_{\sSK}(\bA) \, .\] 
Therefore, for all $\beta<1$, $\P$ and $\Q$ are mutually contiguous. 
(Namely, for a sequence of events $\cE_n$, $\lim_{n\to\infty}\P(\cE_n)=0$ if and only if
$\lim_{n\to\infty}\Q(\cE_n)=0$.)
\end{proposition}
\begin{proof}
Fix $\bx_0 \in \R^n$. We first calculate the density of the process $\by(t) = t\bx_0 + \bB(t)$ with respect to Brownian motion. Let $\W$ be the Wiener measure on $C([0,T], \R^n)$. We obtain by Girsanov's theorem that
\begin{equation} \label{eq:density_girsanov}
\frac{ \rmd \P( \,\cdot\, | \bx_0)}{ \rmd \W}(\by) = e^{ \langle \bx_0 , \by(T)\rangle - T\|\bx_0\|^2/2} \, .
\end{equation}
Notice that the above density only depends on the endpoint $\by(T)$ of the process $\by$. From this,
 we obtain an explicit formula for the density of $\P$ with respect to $(\rmd \bA) \times  \W$:
\begin{equation} 
\P(\rmd\bA, \rmd\by) = \frac{1}{Z_{\pl}} \Big(\int \exp\Big\{-\frac{n}{4} \Big\|\bA - \frac{\beta \bx_0\bx_0^\top}{n}\Big\|_F^2  + \langle \bx_0 , \by(T)\rangle - \frac{T}{2}\|\bx_0\|^2\Big\} \onu(\rmd \bx_0) \Big)\, \rmd \bA \, \W(\rmd \by) \, ,
\end{equation} 
where $Z_{\pl} = \int e^{-n\|\bA\|_F^2/4}\rmd \bA$ is given in Eq.~\eqref{eq:Zpl}.

Next we derive a similar formula for $\Q$. Fix a matrix $\bA \in \R^{n \times n}_{\ssym}$ and let $\by$ be the solution to the SDE in~\eqref{eq:Q}. 
Let $(\bar{\bB}(t))_{t \ge 0}$ be another standard Brownian motion in $\R^n$, 
and consider the process $\bar{\by} = (\bar{\by}(t))_{t \in [0,T]}$ defined by 
\begin{equation}\label{eq:ybar}
\bar{\by}(t) = t \bx +  \bar{\bB}(t) ~~~~ \mbox{where}~~ \bx \sim \mu_{\bA} \mbox{ independently of }\bar{\bB}\, . 
\end{equation} 
Then, there exists another Brownian motion $(\bW(t))_{t\ge 0}$ adapted to
 the filtration $(\F_t = \sigma(\bar{\by}(s) : s\le t))_{t \in [0,T]}$ such that
  $\rmd \bar{\by}(t) = \m_{\bA,\bar{\by}(t)} \rmd t + \rmd \bW(t)$ for all $t \in [0,T]$.
  This is stated as Theorem 7.12 of~\cite{liptser1977statistics}, and can be proved directly applying
  Levy's characterization of Brownian motion to the process
  $\bar{\by}(t) -\int_0^t \m(\bA,\bar{\by}(s)) \rmd s$. 
  
   Therefore, the processes $\bar{\by}$ and $\by$ share the same law conditional on $\bA$.
    Since we computed the law of $\bar{\by}$ in~\eqref{eq:density_girsanov}, we obtain 
\begin{equation} 
\Q(\rmd\bA, \rmd\by) = \frac{1}{Z_{\pl}} \Big(\int \exp\Big\{-\frac{n}{4} \big\|\bA\big\|_F^2 + \langle \bx , \by(T)\rangle - \frac{T}{2}\|\bx\|^2\Big\} \mu_{\bA}(\rmd \bx) \Big) \, \rmd \bA \, \W(\rmd \by) \, ,
\end{equation}
where $Z_{\pl}$ is as above.
Since $\mu_{\bA}(\rmd \bx) = Z_{\sSK}(\bA)^{-1} e^{\beta\langle \bx, \bA \bx \rangle/2 -\beta^2n/4} \onu(\rmd\bx)$, we obtain after simplification 
\begin{align} 
 \frac{\rmd \P}{\rmd \Q}(\bA, \by) = Z_{\sSK}(\bA) \, .
\end{align} 
Mutual contiguity follows from Theorem~\ref{thm:ALR} and Le Cam's first lemma.
\end{proof}

Therefore, for the remainder of the proof of Theorem \ref{thm:main},
 we work under the ``planted'' distribution $\P$. All results proven under $\P$ transfer to $\Q$ by contiguity.

%%%%%%%%%%%%%%%%
\subsection{Approximate Message Passing}
\label{sec:AMP}

In this section we analyze the AMP iteration of Algorithm~\ref{alg:Mean},
recalled here for the reader's convenience:
\begin{align}
\hm^{-1} &= \bz^{0}= 0,\nonumber\\
\hm^{k} &= \tanh(\bz^{k} ) , ~~~~~~~ \sb_{k}= \frac{\beta^2}{n}\sum_{i=1}^n 
\big(1-\tanh^2(z^{k}_i) \big)\,  ~~~~~~~  \forall k\ge 0\, ,\label{eq:AMPreminder}\\
\bz^{k+1} &= \beta \bA \hm^{k} + \by - \sb_{k} \hm^{k-1}\, .\nonumber
\end{align}
When needed, we will specify the dependence on $\bA,\by$ by writing 
$\hm^k = \hm^k(\bA,\by)=\AMP(\bA,\by;k)$
and  $\bz^k = \bz^k(\bA,\by)$. Throughout this section $(\bA,\by)\sim \P$ will be distributed
according to the planted model introduced above.

Our analysis will be based on the general state evolution result of 
\cite{bayati2011dynamics,javanmard2013state},
which implies the following asymptotic characterization for the iterates.  Set $\gamma_0(\beta,t)=0,\Sigma_{0,i}(\beta,t)=0$ and recursively define
\begin{align}
    \label{eq:AMP-nu}
    \gamma_{k+1}(\beta,t)
    &=
    \beta^2\cdot \E\left[\tanh\left(\gamma_k(\beta,t)+t+G_k\right)\right] \, ,
    \\
    \label{eq:AMP-sig}
    \Sigma_{k+1,j+1}(\beta,t)
    &= 
    \beta^2\cdot\E\left[\tanh\left(\gamma_k(\beta,t)+t+G_k\right)\tanh\left(\gamma_j(\beta,t)+t+G_j\right)\right] \, ,
\end{align}
where $(\bG_{j})_{\j\le k}$ are jointly Gaussian, with zero mean and covariance 
$\bSigma_{\le k}+t\bfone\bfone^\top$, $\bSigma_{\le k}:=(\Sigma_{ij})_{i,j\le k}$.

\begin{proposition}[Theorem 1 of \cite{bayati2011dynamics}]\label{prop:state_evolution}
For $(\bA,\by)\sim \mathbb \P$ and any $k\in\mathbb Z_{\ge 0}$, the empirical distribution of the 
coordinate of the AMP iterates converges almost surely in $W_2(\mathbb R^{k+2})$ as follows:
\begin{align}
    &\frac{1}{n}\sum_{i=1}^n \delta_{(z^1_i,\cdots,z^k_i,x_i,y_i)}
    \xrightarrow[n \to \infty]{W_2}
    \cL\left(\bgamma_{\le k}(\beta,t) X+\bG+Y\bfone,X,Y\right) \, ,\\
    &\bgamma_{\le k}(\beta,t) = \big(\gamma_1(\beta,t),\dots,\gamma_k(\beta,t)\big)\, ,
    \;\;\;\;\; \bG \sim \normal(0,\bSigma_{\le k})\, .
\end{align}
On the right-hand side, $X$ is uniformly random in $\{-1,+1\}$,  $Y=tX+\sqrt{t}W$ where
 $W\sim\normal(0,1)$ and $X,\bG,W$ are mutually independent.
\end{proposition}
\begin{remark}
This specific statement follows from \cite[Theorem 1]{bayati2011dynamics} by a change of variables,
as in \cite{deshpande2017asymptotic} or \cite{montanari2021estimation}. 
\end{remark}

As in \cite[Eqs.\ (69,70)]{deshpande2017asymptotic} we argue that the state evolution equations
 \eqref{eq:AMP-nu}, \eqref{eq:AMP-sig} take a simple form thanks to our specific choice of AMP 
 non-linearity $\tanh(\cdot)$. It will be convenient to use the notations
\begin{align*}
    \wtgamma_{k}(\beta,t)&=\gamma_{k}(\beta,t)+t \, ,\\
    \wtSigma_{k,j}(\beta,t)&=\Sigma_{k,j}(\beta,t)+t\, .
\end{align*}

\begin{proposition}
\label{prop:bayes-opt-SE}
For any $t\in\mathbb R_{\geq 0}$ and $k,j\in\mathbb Z_{\geq 0}$,
\begin{align*}
    \Sigma_{k,j}(\beta,t)=\gamma_{k\wedge j}(\beta,t)\, , ~~~\mbox{and}~~~ 
     \wtSigma_{k,j}(\beta,t)=\wtgamma_{k\wedge j}(\beta,t) \, .
\end{align*}
\end{proposition}

\begin{proof}
The two claims are equivalent and we proceed by induction. The base case $k=0$ holds by 
definition, so we may assume $\Sigma_{i,j}(\beta,t)=\gamma_{i\wedge j}(\beta,t)$ for $i,j\le k-1$. Set
 $Z_{j}=\gamma_{j}X+\wtG_{j}$ where $\wtbG \sim \normal(0,\widetilde{\bSigma}_{\le k-1})$. 
 Note that, by the induction hypothesis,  $Z_{k-1}$ is a sufficient statistic for $X$ given 
 $(Z_j)_{j\le k-1}$.
 Using Bayes' rule, and writing $\wtsigma_{k-1}^2:=\wtSigma_{k-1,k-1}$, one easily computes
\begin{align*}
    \mathbb E[X|Z_{k-1}]
    &=
    \frac{e^{\wtgamma_{k-1}Z_{k-1}/\wtsigma_{k-1}^2}-e^{-\wtgamma_{k-1}Z_{k-1}/\wtsigma_{k-1}^2}}{e^{\wtgamma_{k-1}Z/\wtsigma_{k-1}^2}+e^{-\wtgamma_{k-1}Z/\wtsigma_{k-1}^2}}
    =\tanh(Z) \, .
\end{align*}
Therefore using Eq.~\eqref{eq:AMP-nu}, the fact that $\tanh$ is an odd function and $WX \stackrel{\rmd}{=} W$,
\begin{align*}
\wtSigma_{k,j} &= 
\E\big[\E[X|Z_{k-1}]\E[X|Z_{j-1}] \big]\\
& \stackrel{(a)}{=} \E\big[X\E[X|Z_{j-1}] \big]\\
&=\E\left[ X  \tanh(\wtgamma_{j-1}X+\wtsigma_{j-1}^2W)\right]\\
&=\E\left[ \tanh(\wtgamma_{j-1}+\wtsigma_{j-1}^2W)\right] =\gamma_j\, ,
\end{align*}
where in step $(a)$ we crucially used the sufficient statistic property.
This completes the inductive step and hence the proof.
\end{proof}

Define the function $\mmse:\mathbb R\to\mathbb R$ given by
\[
    \mmse(\gamma)\equiv 1-\E\big[\tanh(\gamma+\sqrt{\gamma}W)^2\big]=1-\E\big[\E[X|\gamma X+\sqrt{\gamma}W]^2\big].
\]
It follows from Proposition~\ref{prop:bayes-opt-SE} that \eqref{eq:AMP-nu} and \eqref{eq:AMP-sig} can be expressed just in terms of the sequence $\gamma_k(\beta,t)$ defined by $\gamma_0(t)=0$ and the recursion
\begin{equation}
\label{eq:gamma-recusion}
    \gamma_{k+1}(\beta,t)=\beta^2\big(1-\mmse(\gamma_k(\beta,t)+t)\big).
\end{equation}
Note that $\gamma_k(\beta,t)$ depends also on $\beta$, which is usually treated as constant. 
The following result details some useful properties of $\mmse$.
We note that we do not assume $\beta<1$ unless explicitly stated, which is important for Proposition~\ref{prop:bayes-opt-MMSE}.
\begin{lemma}[Lemma 6.1 of \cite{deshpande2017asymptotic}]
\label{lem:properties}
The following properties hold, where $\{\gamma_k(\beta,t)\}_{k\geq 1}$ is as defined by \eqref{eq:gamma-recusion}.
\begin{enumerate}[label=(\alph*)]
    \item 
    \label{it:mmse-decreasing-convex} 
    $\mmse$ is differentiable,  strictly decreasing, and convex in $\gamma\in\mathbb R_{\geq 0}$.
    \item
    \label{it:mmse-boundary-values} 
    $\mmse(0)=1$, $\mmse'(0)=-1$ and $\lim_{\gamma\to\infty}\mmse(\gamma)=0$.
    \item\label{it:mmseExistence} For $t\geq 0$ there exists a non-negative solution $\gamma_*=\gamma_*(\beta,t)$ to the fixed point equation
     \begin{equation}
        \label{eq:gamma*}
        \gamma_*=\beta^2(1-\mmse(\gamma_*+t))\, .
    \end{equation}
The solution to this equation is unique for all $t>0$, and $\lim_{k\to\infty}\gamma_k(\beta,t)=\gamma_*(\beta,t)$.
    \item \label{it:mmseSolSmooth} The function $(\beta,t)\mapsto \gamma_*(\beta,t)$ is differentiable for $t>0$ and $\beta<1$. 
    \item For all $\beta<1$ and $t>0$, 
    \begin{equation}
    \label{eq:uniform-gamma-limit}
        %\lim_{k\to\infty} \inf_{t>0}\frac{\gamma_k(\beta,t)}{\gamma_*(\beta,t)}=1.
        1-\beta^{2k} \le \frac{\gamma_{k}(\beta,t)}{\gamma_*(\beta,t)} \le 1 \, .
    \end{equation}
    \item 
    \label{it:gamma*-near0}
    For $\beta<1$ and $\T>0$, there exist constants $c(\beta,\T),C(\beta,\T)\in (0,\infty)$
    such that, for all $t\in (0,\T]$,
    \begin{equation}
    \label{eq:gamma*-near0}
        %c(\beta,\T)\leq \frac{\gamma_1(\beta,t)}{t}\leq \frac{\gamma_*(\beta,t)}{t}\leq C(\beta,\T).
        c(\beta,\T)\leq \frac{\gamma_*(\beta,t)}{t}\leq C(\beta,\T) \, .
    \end{equation}
    \item
    \label{it:gamma*-Lipschitz}
    For $\beta<1$ and any $t_1,t_2\in (0,\infty)$,
    \begin{equation}
    \label{eq:gamma*-Lipschitz}
        \gamma_*(\beta,t_1)-\gamma_*(\beta,t_2)\leq \frac{\beta^2}{1-\beta^2} |t_1-t_2|.
    \end{equation}
\end{enumerate}
\end{lemma}

\begin{proof}
 Lemma 6.1 in \cite{deshpande2017asymptotic} proves that $\gamma\mapsto\mmse(\gamma)$ 
 is differentiable,  strictly decreasing, and convex in $\gamma\in\mathbb R_{\geq 0}$
 (note that the statement of that Lemma does not claim differentiability, but this is actually proved 
 there by a simple application of dominated convergence). This proves point 
 \ref{it:mmse-decreasing-convex}.
 
  Point \ref{it:mmse-boundary-values} follows by a direct calculation, cf. \cite{deshpande2017asymptotic}.
  Indeed,  by Stein's lemma (Gaussian integration by parts), with $Z=\gamma+W\sqrt{\gamma}$,
\begin{align*}
    -\mmse'(\gamma)&=\frac{\rmd~}{\rmd\gamma}\E[\tanh(\gamma+W\sqrt{\gamma})^2]\\
    &=\E[2\tanh(Z)\tanh'(Z)+\tanh'(Z)^2 +\tanh(Z)\tanh''(Z)]
\end{align*}
Evaluating at $\gamma=0$ shows
\[
    \mmse'(0)=-1.
\]
Also, dominated convergence yields the desired limit values.

Point \ref{it:mmseExistence} will follow from the above
monotonicity and convexity properties.
Indeed, we have $0< \beta^2(1-\mmse(t))$ while $\gamma> \beta^2(1-\mmse(\gamma+t))$ for all sufficiently large $\gamma$.  
Thus a solution $\gamma_*$ exists by the intermediate value theorem, and since $\gamma\mapsto \beta^2(1-\mmse(\gamma+t))$ is concave, this solution is unique.
Point \ref{it:mmseSolSmooth} follows from the implicit function theorem.

We are left with the task of proving 
\eqref{eq:uniform-gamma-limit}, \eqref{eq:gamma*-near0} and \eqref{eq:gamma*-Lipschitz},
which are not given in  \cite{deshpande2017asymptotic}.

Define
\[
    f_t(\gamma)\equiv \beta^2(1-\mmse(\gamma+t))
\]
so that $f_t(\gamma_k(\beta,t))=\gamma_{k+1}(\beta,t)$.
By point~\ref{it:mmse-boundary-values}, $f_t(0)\geq 0$. By point~\ref{it:mmse-decreasing-convex}, $f_t(\cdot)$ is increasing and concave. Combined with the computation above, we conclude that $f_t'(\gamma)\in [0,\beta^2]$ for all $\gamma\geq 0$. By the mean value theorem, it follows that for $\gamma<\gamma_*$,
\begin{equation}
\label{eq:gamma-linear-convergence}
    0
    \leq 
    \gamma_*(\beta,t)-f_t(\gamma)
    =
    f(\gamma_*(\beta,t))-f_t(\gamma)
    \leq 
    \beta^2(\gamma_*(\beta,t)-\gamma) \, .
\end{equation}
Setting $\gamma=\gamma_{j}(\beta,t)$, we obtain
\[
    0 \le \frac{\gamma_*(\beta,t)-\gamma_{j+1}(\beta,t)}{\gamma_*(\beta,t)-\gamma_{j}(\beta,t)}\leq \beta^2 .
\]
Multiplying for $j\in \{0,\dots,k-1\}$, we find   
\[0\le \frac{\gamma_*(\beta,t)-\gamma_k(\beta,t)}{\gamma_*(\beta,t)}\leq \beta^{2k} \, ,\]
or, 
\[\frac{\gamma_{k}(\beta,t)}{\gamma_*(\beta,t)} \in \left[ 1-\beta^{2k}, 1\right] \, ,\]  
which proves \eqref{eq:uniform-gamma-limit}. 

To prove \eqref{eq:gamma*-near0}, note that we just showed 
\[
    \frac{\gamma_1(\beta,t)}{\gamma_*(\beta,t)}\in \left[ 1-\beta^2, 1\right].
\]
Therefore it suffices to show that
\begin{equation}
\label{eq:gamma1/t}
    c(\beta,\T)\leq\frac{\gamma_1(\beta,t)}{t}\leq C(\beta,\T),\quad t\in (0,\T] \, .
\end{equation}
By definition, $\gamma_1(\beta,t)=\beta^2(1-\mmse(t))$. Thus \eqref{eq:gamma1/t} follows from the fact that $\mmse(0)=1$, $\mmse'(0)=-1$, and $\mmse:\mathbb R_{\geq 0}\to [0,1]$ is convex and strictly decreasing. In fact we have $\gamma_1(\beta,\T)/\T \le \gamma_1(\beta,t)/t \le \beta^2$ for all $t \in (0,\T]$.

Finally, we prove \eqref{eq:gamma*-Lipschitz}. Since $|\mmse'(t)|\leq 1$ for all $t\geq 0$ we find that for $t_1,t_2\geq 0$,
\begin{align*}
    |\gamma_*(\beta,t_1)-\gamma_*(\beta,t_2)|
    &=
    \beta^2 |\mmse(\gamma_*(\beta,t_1)+t_1)-\mmse(\gamma_*(\beta,t_2)+t_2)|\\
    &\leq 
    \beta^2 |\gamma_*(\beta,t_1)-\gamma_*(\beta,t_2)| +\beta^2|t_1-t_2|.
\end{align*}
Rearranging, we obtain
\[
    \frac{|\gamma_*(\beta,t_1)-\gamma_*(\beta,t_2)|}{|t_1-t_2|}\leq \frac{\beta^2}{1-\beta^2}.
\]

\end{proof}

For $(\bA,\by)\sim \P$ and $\bx\sim\mu_{\bA,\by(t)}$, define
\begin{equation}\label{eq:MSE_AMP}
    \MSE_{\AMP}(k;\beta,t)=
    \plim_{n\to\infty}
    \frac{1}{n}
    \E \big\|\bx-\hm^k(\bA,\by(t)) \big\|_2^2 \, ,
    \;\;\;\hm^k(\bA,\by(t)):=\AMP(\bA,\by(t);k)\,,
\end{equation}
where the limit is guaranteed to exist by Proposition~\ref{prop:state_evolution}. 

\begin{lemma}
\label{lem:MSE-k}
We have
\begin{align*}
    \MSE_{\AMP}(k;\beta,t)
    %&=1-\plim_{n\to\infty}\frac{1}{n} \big\|\AMP_{\beta}(\bA,\by(t);k)\big\|_2^2\\
    &=1-\frac{\gamma_{k+1}(\beta,t)}{\beta^2} \, .
\end{align*}
In particular,
\begin{align*}
\lim_{k\to\infty}\MSE_{\AMP}(k;\beta,t)&=1-\frac{\gamma_*(\beta,t)}{\beta^2} \, .
\end{align*}
\end{lemma} 
\begin{proof}
By state evolution
\begin{align*}
    \MSE_{\AMP}(k;\beta,t)&=\plim_{n\to\infty}\frac{1}{n}\E\big\|\hm^k(\bA,\by(t))-\bx\big\|_2^2\\
    &= \E\big[\big(\tanh(\gamma_k X+\sigma_k W+Y)-X\big)^2\big]\\
    &= \E\big[\big(\tanh(\wtgamma_k X+\wtsigma_k W)-X\big)^2\big]\\
    &= 1 - 2 \E[\tanh(\wtgamma_k X+\wtsigma_k W)X]+\E[\tanh(\wtgamma_k X+\wtsigma_k W)^2]\\
    &= 1 - 2 \gamma_{k+1}/\beta^2 + \sigma_{k+1}^2/\beta^2 \\
    &=1 - \gamma_{k+1}/\beta^2,
\end{align*}
where the last line follows from Proposition~\ref{prop:bayes-opt-SE}.
\end{proof}

We next show that,
for any $t>0$, the mean square error achieved by AMP is the same as the Bayes optimal error,
i.e., the mean squared error achieved by the posterior expectation $\m(\bA,\by(t))$.
\begin{proposition}
\label{prop:bayes-opt-MMSE}
Fix $\beta<1$ and $t\ge0$. We have
\begin{equation}
\label{eq:MMSE}
    \lim_{n\to\infty}\frac{1}{n}\E\left[\big\|\bx-\m(\bA,\by(t))\big\|_2^2\right]=
    \frac{\gamma_{*}(\beta,t)}{\beta^2}\, .
\end{equation}
\end{proposition}
\begin{proof}
The proof is an adaptation from \cite{deshpande2017asymptotic}, which we will present
succinctly. 
In fact we show the result for all $\beta$ in the planted model.
%Alternative proofs could be given using the results of \cite{}. \amcomment{Add refs}

Let $I(X;Y)$ denote the mutual information between random variables $X,Y$
on the same probability space. Letting $X\sim\Unif(\{-1,+1\})$ independent of 
$W\sim\normal(0,1)$,  define the function 
\begin{align}
\info(\gamma) &:= I\big(X;\gamma X+\sqrt{\gamma} W\big) \\
& = \gamma -\E\log\cosh\big(\gamma +\sqrt{\gamma} W)\big)\, .
\end{align}
We also define the function
\begin{align}
\Psi(\gamma;\beta,t):= \frac{\beta^2}{4}+\frac{\gamma^2}{4\beta^2}-\frac{\gamma}{2}+\info(\gamma+t)\, . 
\end{align}
As in \cite{deshpande2017asymptotic}, it is easy to check that 
$\partial_{\gamma}\Psi(\gamma_*(\beta,t);\beta,t)=0$ and,
using the continuity of $(\beta,t)\mapsto \gamma_*(\beta,t)$, 
\begin{align}
\frac{\de\phantom{\beta^2}}{\de(\beta^2)}\Psi(\gamma_*(\beta,t);\beta,t)&=
\frac{1}{4} \Big(1-\frac{\gamma_*(\beta,t)^2}{\beta^4}\Big)\, ,\label{eq:DerivativePsiBeta}\\
\frac{\de\phantom{t}}{\de t}\Psi(\gamma_*(\beta,t);\beta,t)&=
\frac{1}{2} \Big(1-\frac{\gamma_*(\beta,t)}{\beta^2}\Big)\, .
\label{eq:DerivativePsit}
\end{align}
We further note that by the de Brujin identity (also known as I-MMSE relation \cite{Guo05mutualinformation})
\begin{align}
\frac{\de\phantom{\beta^2}}{\de(\beta^2)}I(\bx;\bA(\beta),\by(t)) &=
\frac{1}{4n}\E\left[\big\|\bx\bx^{\top}-\E\{\bx\bx^{\top}|\bA(\beta),\by(t)\}\big\|_F^2\right] \, ,
\label{eq:DerivativeIBeta}\\
\frac{\de\phantom{t}}{\de t}I(\bx;\bA(\beta),\by(t)) &=
\frac{1}{2}\E\left[\big\|\bx-\E\{\bx|\bA(\beta),\by(t)\}\big\|_2^2\right]\, .
\label{eq:DerivativeIt}
\end{align}
Here we write $\bA=\bA(\beta)$ to emphasize the dependence upon $\beta$.
Using Eqs.~\eqref{eq:DerivativePsiBeta} and \eqref{eq:DerivativeIBeta}, 
we have
\begin{align*}
\log 2-\info(t) &=\lim_{\beta\to\infty}\lim_{n\to\infty}\frac{1}{n}\big[I(\bx;\bA(\beta),\by(t))-I(\bx;\bA(0),\by(t))\big]\\
&=\lim_{n\to\infty} \int_{0}^{\infty}\frac{1}{4n}\E\left[\big\|\bx\bx^{\top}-\E\{\bx\bx^{\top}|\bA(\beta),\by(t)\}\big\|_F^2\right] \,\de\beta^2\\
&\le \lim_{k\to\infty}\lim_{n\to\infty}
\int_{0}^{\infty}\frac{1}{4n}\E\left[\big\|\bx\bx^{\top}-\hm^k(\bA(\beta),\by(t))\hm^k(\bA(\beta),\by(t))^\top\big\|_F^2\right] \,\de\beta^2\\
& =  \lim_{k\to\infty}\int_{0}^{\infty} \frac{1}{4} \Big(1-\frac{\gamma_k(\beta,t)^2}{\beta^4}\Big)\,\de\beta^2\\
& =  \int_{0}^{\infty} \frac{1}{4} \Big(1-\frac{\gamma_*(\beta,t)^2}{\beta^4}\Big)\,\de\beta^2\\
& = \lim_{\beta\to\infty}\big[\Psi(\gamma_*(\beta,t);\beta,t)-\Psi(\gamma_*(0,t);0,t)\big]\, .
\end{align*}
(The exchanges of limits are justified by dominated convergence.)

Finally, a direct calculation reveals that 
$\lim_{\beta\to\infty}\big[\Psi(\gamma_*(\beta,t);\beta,t)-\Psi(\gamma_*(0,t);0,t)\big]=\log(2)-\info(t)$
and therefore equality holds at each of the steps above. We deduce that 
$\lim_{n\to\infty}n^{-1}I(\bx;\bA(\beta),\by(t))=\Psi(\gamma_*(\beta,t);\beta,t)$.

Using this fact, together with Eqs.~\eqref{eq:DerivativePsit}, \eqref{eq:DerivativeIt}
and the fact that the right hand sides of these equations are monotone decreasing in $t$,
we get that the following holds for almost every $t>0$:
\begin{align}
\lim_{n\to\infty}\frac{1}{n}\E\left[\big\|\bx-\E\{\bx|\bA(\beta),\by(t)\}\big\|_2^2\right]
= 1-\frac{\gamma_*(\beta,t)}{\beta^2}\, .
\end{align}
This coincides with the claim \eqref{eq:MMSE}, and actually holds for every $t>0$ since 
the right-hand side of Eq.~\eqref{eq:MMSE} is continuous in $t>0$ by Lemma \ref{lem:properties}.
\end{proof}

It follows that AMP approximately computes the posterior mean 
$\m(\bA,\by(t))$ 
in the following sense.
\begin{proposition}
\label{prop:amp-posterior-mean}
Fix $\beta<1$, $\T>0$ and let $t \in (0,\T]$. Recalling that 
$\hm^k(\bA,\by(t)):=\AMP(\bA,\by(t);k)$ denotes the AMP estimate after $k$ iterations, 
and that $\bz^k$ is defined by Eq.~\eqref{eq:AMPreminder}, we have
\begin{align}
\label{eq:m-converge}
    &\lim_{k\to\infty}\sup_{t\in (0,\T)}\plim_{n\to\infty} 
    \frac{\|\m(\bA,\by(t))-\hm^k(\bA,\by(t))\|_2}{\|\m(\bA,\by(t))\|_2}=0 \, .
    \end{align}
    Moreover
    \begin{align}
\label{eq:z-converge}
    &\lim_{k\to\infty}\sup_{t\in (0,\T)}\plim_{n\to\infty} \frac{\|\bz^{k+1}-\bz^k\|}{\|\bz^k\|}=0 \, .
\end{align}
\end{proposition}
\begin{remark}
A somewhat similar result has recently been proved by Chen and Tang~\cite{chen2021convergence} where the external field vector $\by(t)$ is replaced by a multiple of the all-ones vector $h \mathbf{1}$, for any pair $(\beta,h)$ for which a certain condition of uniform concentration of the overlap between two independent draws from the measure $\mu_{\bA,h\mathbf{1}}$ holds. In our setting, we are concerned with a different family of external fields, namely the ones generated by the stochastic localization process~\eqref{eq:GeneralSDE}. The argument, which proceeds via the planted model, does not require the uniform concentration condition.   
\end{remark}
\begin{proof}
Throughout this proof we write $\by$ instead of $\by(t)$ for ease of notation. 
To show Eq.~\eqref{eq:m-converge}, observe that the bias-variance decomposition yields
(recalling the definition $\MSE_{\AMP}(\;\cdot\;)$ in Eq.~\eqref{eq:MSE_AMP})
\begin{align*}
    \MSE_{\AMP}(k;\beta,t)
    &=
    \plim_{n\to\infty}\left\{
    \frac{1}{n}
    \E
    \Big[
    \big\|\hm^k(\bA,\by)-\m(\bA,\by)\big\|_2^2\Big]+
    \frac{1}{n}
    \E
    \Big[\big\|\bx- \m(\bA,\by)\big\|_2^2
    \Big]\right\}.
\end{align*}
Using Lemma \ref{lem:MSE-k} for the left-hand side and
Proposition~\ref{prop:bayes-opt-MMSE} for the second step the second term on the right-hand side, 
we get
\begin{equation}\label{eq:limgamma}
    \plim_{n\to\infty}
    \frac{1}{n}
    \E\left[
    \big\|\hm^k(\bA,\by)-\m(\bA,\by) \big\|_2^2
    \right] = \frac{\gamma_*(\beta,t)-\gamma_{k+1}(\beta,t)}{\beta^2} \, .
\end{equation}
Claim \eqref{eq:m-converge} now follows by combining Eq.~\eqref{eq:limgamma} 
with  Eqs.~\eqref{eq:uniform-gamma-limit} and \eqref{eq:gamma*-near0}
of Lemma \ref{lem:properties}.

Finally, Eq.~\eqref{eq:z-converge}
is an immediate consequence of Proposition~\ref{prop:state_evolution}
and Proposition \ref{prop:bayes-opt-SE}. Indeed, by  Proposition~\ref{prop:state_evolution},
we have
\begin{align}
    \plim_{n\to\infty} \frac{1}{n}\big\|\bz^k\big\|_2^2 
    &=  \E\big[(\gamma_k X + G_k+ Y)^2\big] =  (\gamma_k+t)^2 + \gamma_k+t \, ,\\
      \plim_{n\to\infty} \frac{1}{n}\big\|\bz^{k+1}-\bz^k\big\|_2^2 
    &=  \E\big[\big((\gamma_{k+1}-\gamma_k) X +G_{k+1}-G_k)^2\big]\\
    & =  (\gamma_{k+1}-\gamma_k)^2 + (\Sigma_{k+1,k+1}-2\Sigma_{k,k+1}+\Sigma_{k,k}) \\
    & =  (\gamma_{k+1}-\gamma_k)^2 + (\gamma_{k+1}-\gamma_k)\, ,
\end{align}
where in the last step we used  Proposition \ref{prop:bayes-opt-SE}.
We therefore obtained
we have
\begin{align}
    \plim_{n\to\infty} \frac{\|\bz^{k+1}-\bz^k\big\|_2^2}{\|\bz^k\|_2^2} =
     \frac{(\gamma_{k+1}-\gamma_k)^2 + (\gamma_{k+1}-\gamma_k)}{(\gamma_k+t)^2 + \gamma_k+t}\, .
     \end{align}	
Hence Eq.~\eqref{eq:z-converge} also follows from Eq.~\eqref{eq:uniform-gamma-limit}.
\end{proof}

We conclude this subsection with a lemma controlling the regularity of the 
posterior  path $t\mapsto \m(\bA,\by(t))$, which will be useful later.
\begin{lemma}
\label{lem:uniform-path}
Fix $\beta<1$ and $0\leq t_1<t_2\leq \T$. Then 
\begin{align}
    \plim_{n\to\infty}\sup_{t\in [t_1,t_2]} \frac{1}{n} \big\|\m(\bA,\by(t))-\m(\bA,\by(t_1))\big\|_2^2&= 
    \plim_{n\to\infty} \frac{1}{n} \big\|\m(\bA,\by(t_2))-\m(\bA,\by(t_1)) \big\|_2^2\\
    &=\frac{\gamma_*(\beta,t_2)-\gamma_*(\beta,t_1)}{\beta^2} \, .\label{eq:uniform-path}
\end{align}
\end{lemma}
\begin{proof}
We will exploit the fact that $(\m(\bA,\by(t)))_{t\ge 0}$ is a martingale,
as a consequence of Lemma~\ref{prop:stochloc1} (with $\varphi:\mathbb R^n\to\mathbb R^n$ given by
 $\varphi(\bx)=\bx$).
 
Using Proposition \ref{prop:bayes-opt-MMSE}, we obtain, for any $t_1<t_2$
\begin{align*}
    \lim_{n\to\infty} 
     \frac{1}{n}\E\big[\big\|\m(\bA,\by(t_2))-\m(\bA,\by(t_1))\big\|_2^2 \big]
    &=  \plim_{n\to\infty} 
     \frac{1}{n}\Big\{\E\big[\big\|\bx-\m(\bA,\by(t_1))\big\|_2^2 \big]
     -\E\big[\big\|\bx-\m(\bA,\by(t_1))\big\|_2^2 \big]\Big\}\\
     & = 
    \frac{\gamma_*(\beta,t_2)-\gamma_*(\beta,t_1)}
    {\beta^2} \, ,
\end{align*}
where the first equality uses the fact that $\E\{\m(\bA,\by(t_2))|\bA,\by(t_1)\}=
\m(\bA,\by(t_1))$. By Lemma \ref{prop:amp-posterior-mean}, we have, with high probability,
$\|\m(\bA,\by(t_i))-\hm^k(\bA,\by(t))\|_2^2/n\le \eps_k$, for some deterministic constants
$\eps_k$ so that $\eps_k\to 0$ as $k\to\infty$. As a consequence
\begin{align}
    \plim_{n\to\infty} 
     \frac{1}{n}\big\|\m(\bA,\by(t_2))-\m(\bA,\by(t_1))\big\|_2^2
    &=  \frac{\gamma_*(\beta,t_2)-\gamma_*(\beta,t_1)}
    {\beta^2}\, .\label{eq:MM-gamma}
\end{align}
    
Now, since $t\to \m_{\bA,\by(t)}$ is a bounded martingale, it follows that,
for any fixed constant $c$, the process 
\begin{equation}
 Y_{n,t} := \max(M_{n,t} - c,0)=(M_{n,t}-c)_+ \, , ~~~\mbox{where}~~~ M_{n,t} := 
 \frac{1}{\sqrt{n}} \big\|\m(\bA,\by(t))-\m(\bA,\by(t_1))\big\|_2 \, ,
\end{equation}
is a positive bounded submartingale for $t\ge t_1$.
 Therefore by Doob's maximal inequality~\cite{durrett2019probability},
\begin{equation}
    \P\Big(\sup_{t\in [t_1,t_2]}Y_{n,t} \geq a \Big)
    \leq 
    \frac{1}{a} \E\big[Y_{n,t_2} \big] \le \frac{1}{a} \E\big[Y_{n,t_2}^2 \big]^{1/2} \, ,
\end{equation}
for any $a>0$. We choose $c =  \sqrt{\gamma_*(\beta,t_2)-\gamma_*(\beta,t_1)}/\beta$. 
By \eqref{eq:MM-gamma}, we have
\[
    \plim_{n\to\infty} M_{n,t_2}^2 = \frac{\gamma_*(t_2)-\gamma_*(t_1)}{\beta^2} =c^2\, ,
\]
and therefore, since $M_{n,t}$ is bounded, for any fixed $a>0$
\begin{align*}
\lim_{n\to\infty}\P\Big(\sup_{t\in [t_1,t_2]}M_{n,t} \geq c + a \Big) &\le 
\lim_{n\to\infty}\P\Big(\sup_{t\in [t_1,t_2]}Y_{n,t} \geq a \Big) \\
&\le \frac{1}{a}\lim_{n\to\infty}\E\big[(M_{n,t_2}-c)_+^2 \big]^{1/2} = 0\, .
\end{align*}
Together with Eq.~\eqref{eq:MM-gamma}, this yields 
\[
    \plim_{n\to\infty}\sup_{t\in [t_1,t_2]}M_{n,t}^2=\frac{\gamma_*(t_2)-\gamma_*(t_1)}{\beta^2} \, ,
\]
which coincides with the claim \eqref{eq:uniform-path}.
\end{proof}

%%%%%%%%%%%%%%%
\subsection{Natural Gradient Descent}
\label{sec:ngd}

\begin{algorithm}
\label{alg:NGD}
\DontPrintSemicolon % Some LaTeX compilers require you to use \dontprintsemicolon instead
\KwIn{Initialization $\bu^0\in \R^n$, data $\bA\in \mathbb R^{n\times n}$, $\hby\in \mathbb R^n$, step size
 $\eta>0$, $q\in (0,1)$, integer $K>0$.}
$\hm^{+,0} = \tanh(\bu^0)$. \\
\For{$k = 0,\cdots,K-1$} { 
$\bu^{k+1} \leftarrow \bu^k - \eta \cdot\nabla \cuF_{\sTAP}(\hm^{+,k};\by,q)$,  \label{line:explicit-NGD-step} \\
$\hm^{+,k+1} = \tanh(\bu^{+,k+1})$,
}
\Return{$\hm^{+,K}$}\;
\caption{{\sc Natural Gradient Descent on $\cuF_{\sTAP}(\;\cdot\; ;\by,q)$}}
\end{algorithm}

The main objective of this section is to show that
$\cuF_{\sTAP}(\m;\by,q)$ behaves well for $q=q_*(\beta,t)$
and for $\m$ in a neighborhood of $\hm^{K_{\sAMP}}$. Namely it has a unique local minimum 
$\m_* = \m_*(\bA,\by)$ in such a neighborhood,   and
NGD approximates $\m_*$ well for large number of iterations $K$. 
 Crucially, the map 
 $\by \mapsto \m_*$ will be Lipschitz. 
For reference, we reproduce the NGD algorithm as Algorithm \ref{alg:NGD}.
This corresponds to lines \ref{alg:NGD-begin}-\ref{alg:NGD-end} of Algorithm \ref{alg:Mean}.
%
 %Crucially, the map 
 %$\by \mapsto \m_*$ will be Lipschitz. 
%
%
\begin{lemma}
\label{lem:local-landscape}
Let $\beta<\frac{1}{2}$, $c\in (0,1-2\beta)$, and $\T>0$ be fixed.
Then there exists $\eps_0 = \eps_0(\beta,\T)$ such that, for all $\eps\in (0,\eps_0)$
there exists $K_{\sAMP} = K_{\sAMP}(\beta,\T,\eps)$ and 
 $\rho_0 =\rho_0(\beta,\T,\eps)$ such that 
 for all $\rho\in (0,\rho_0)$  there exists  $K_{\sNGD} = K_{\sNGD}(\beta,\T,\eps,\rho)$,
 such that the 
 following holds.
 
 Let $\hm^{\sAMP}= \AMP(\bA,\by(t);K_{\sAMP})$ be the output of the AMP after 
 $K_{\sAMP}$ iterations, when applied to $\by(t)$. Fix $K\ge K_{\sAMP}$.
With probability $1-o_n(1)$ over $(\bA,\by)\sim\P$, 
for all $t\in (0,\T]$ and all $\hby\in B\left(\by(t),c\sqrt{\eps tn} / 4\right)$,
setting $q_{*} := q_{*}(\beta,t)$:
\begin{enumerate}
    \item 
    \label{it:landscape-basic}
    The function 
    \[
        \m\mapsto \cuF_{\sTAP}(\m ; \hby, q_{*})
    \]
    restricted to $B\left(\hm^{\sAMP},\sqrt{\eps tn}\right)\cap (-1,1)^n$ has a unique stationary point 
    \[
        \m_*(\bA,\hby)\in  B\left(\hm^{\sAMP},\sqrt{\eps tn} / 2\right)\cap (-1,1)^n
    \]
    which is also a local minimum. In the case $\hby=\by(t)$, $\m_*(\bA,\by(t))$ 
    also satisfies 
    \[
        \m_*(\bA,\by)\in  B\left(\hm^{k'},\sqrt{\eps tn} / 2\right)\cap (-1,1)^n
    \]
    for all $k'\in [K_{\sAMP},K]$, where $\hm^{k'} = \AMP(\bA,\by(t);k')$.
    \item
    \label{it:landscape-stationary-point-good}
    The stationary point $\m_*(\bA,\hby)$ satisfies
    (recall that $\m(\bA,\by)$ denotes the mean of the Gibbs measure)
    \[
        \big\|\m(\bA,\by)-\m_*(\bA,\by) \big\|_2\leq \rho\sqrt{tn} \, .
    \]
    \item The stationary point $\m_*$ satisfies the following Lipschitz property
    for all  $\hby,\hby'\in B\left(\by(t),c\sqrt{\eps tn} / 4\right)$:
    \label{it:landscape-lipschitz}
    \begin{equation}
    \label{eq:stationary-point-lipschitz}
        \big\|\m_*(\bA,\hby)-\m_*(\bA,\hby')\big\| \leq c^{-1} \|\hby-\hby'\| \, .
    \end{equation}
    \item 
    \label{it:landscape-NGD}
    There exists a learning rate $\eta=\eta(\beta,\T,\eps)$ such that the following holds.
    Let $\hm^{\sNGD}(\bA,\hby)$ be the output of NGD (Algorithm
    \ref{alg:NGD}), when run for $K_{\sNGD}$ iterations with parameter $q_{*}$,
    $\hby$, $\eta$. Assume that the initialization $\bu^0$ satisfies 
    \begin{align}\label{ass:landscape-NGD}
        \big\|\bu^0-\atanh(\hm^{\sAMP})\big\|\leq \frac{c\sqrt{\eps t n}}{200}  \, .
    \end{align}
    Then the algorithm output satisfies 
    \begin{equation}
    \label{eq:landscape-NGD-convergence}
        \big\| \hm^{\sNGD}(\bA,\hby)-\m_*(\bA,\hby) \big\| \leq \rho\sqrt{tn} \, .
    \end{equation}
\end{enumerate}
\end{lemma}
The proof of this lemma is deferred to the appendix. Here we will prove the two
key elements: first that  $\hm^{\sAMP}$ is an approximate stationary point of
 $\cuF_{\sTAP}(\;\cdot\;;\by(t),q_*)$
  (Lemma~\ref{lem:TAP-stationary}), and second that
  $\cuF_{\sTAP}(\;\cdot\;;\hby,q_*)$ is strongly convex in a neighborhood of 
  $\hm^{\sAMP}$ (Lemma~\ref{lem:local-convex}). Let us point out that, in the local convexity 
  guarantee, it is important that the neighborhood has radius 
  $\Theta(\sqrt{tn})$ as $t\to 0$.

We recall below the expressions for the gradient and Hessian of  $\cuF_{\sTAP}(\;\cdot\;;\by,q)$
at
$\m\in (-1,1)^n$:
\begin{align}
    \label{eq:grad-F}
    \nabla\cuF_{\sTAP}(\m;\by,q)&=
    -\beta \bA\m-\by+\atanh(\m)
    +
    \beta^2\left(1-q\right)\m 
    \\
    \label{eq:hess-F}
    \nabla^2\cuF_{\sTAP}(\m;\by,q)&=
    -\beta \bA+\D(\m)
    +
    \beta^2\left(1-q\right)\bI_n \,, \;\; \bD(\m):={\rm diag}\big(\{(1-m_i^2)^{-1}\}_{i\le n}\big).
\end{align}
In \eqref{eq:grad-F}, $\atanh$ is applied coordinate-wise to $\m\in (-1,1)^n$.

For $t >0, k \ge 0$ we let $\hm^k=\AMP(\bA,\by(t);k)$ and define the quantities 
\begin{align}
\label{eq:qk*}
    q_k(\beta,t) &:= \frac{\gamma_{k+1}(\beta,t)}{\beta^2}\, ,
\;\;\;\;\; q_*(\beta,t) :=\frac{\gamma_*(\beta,t)}{\beta^2} \, .
\end{align}
Note that, by  Lemma~\ref{lem:MSE-k}, we have
\begin{align}
\label{eq:qk*-lim}
    q_k(\beta,t) = \plim_{n\to\infty}\frac{\big\|\hm^k\big\|^2}{n} 
\, ,\;\;\;\;\;
  q_*(\beta,t) =\lim_{k\to\infty}q_k(\beta,t) \, .
\end{align}
 We will use the bounds~\eqref{eq:uniform-gamma-limit}, \eqref{eq:gamma*-near0} in Lemma~\ref{lem:properties} several times below, 
 which ensures that $(q_k(\beta,t)/t)\in [c,C]$ holds for constants $c,C>0$ 
 independent of $t\in (0,\T]$ and $k\geq 1$.

\begin{lemma}
\label{lem:TAP-stationary}
Let $\hm^k = \hm^k(\bA,\by(t))$ denote the AMP iterates on input $\bA,\by(t)$. 
Then for any $ \T>0$,
\[
    \lim_{k\to\infty}\sup_{t\in (0,\T]}\sup_{q\in [q_k(\beta,t),q_*(\beta,t)]}\plim_{n\to\infty} \frac{\big\|\nabla \cuF_{\sTAP}(\hm^k;\by(t),q)\big\|}{\sqrt{tn}}=0.
\]
\end{lemma}

\begin{proof}

As in Algorithm \ref{alg:Mean}, let
\[
  \bz^{k+1}=\atanh(\hm^{k+1})=\beta \bA\hm^k+\by-\beta^2\left(1- \frac{1}{n}\big\|\hm^{k}\big\|^2\right) \hm^{k-1}.
\]
Let $q\in [q_k(\beta,t),q_*(\beta,t)]$. Combining the above with 
Eqs.~\eqref{eq:grad-F} and \eqref{eq:qk*-lim} yields
\begin{align*}
    \frac{1}{\sqrt n}\|\nabla \cuF_{\sTAP}(\AMP(\bA,\by;k);\by,q)\|
    &=
    \frac{1}{\sqrt n}\left\|
    -
    \beta \bA\hm^k-\by
    +
    \atanh(\hm^k)
    +
    \beta^2(1-q)\hm^k\right\| 
    \\
    &=
    \frac{1}{\sqrt n}\left\|
    \bz^k-\beta \bA\hm^k-\by
    +
    \beta^2\left(1-q\right)\hm^k
    \right\| 
    \\
    &\leq 
    \frac{1}{\sqrt n}\|\bz^{k+1}-\bz^{k}\|
    +
    \frac{1}{\sqrt n}\left\|\bz^{k+1}
    -
    \beta \bA\hm^k-\by+\beta^2\left(1-q\right)\hm^k\right\|
    \\
    &=\frac{1}{\sqrt n}\big\|\bz^{k+1}-\bz^{k} \big\|
    +\frac{\beta^2}{\sqrt n}\left\| \big(1- \big\|\hm^{k}\big\|^2/n\big) \hm^{k-1}- \left(1-q\right)\hm^k\right\|
    \\
    &\leq 
    \frac{1}{\sqrt n}\big\|\bz^{k+1}-\bz^{k}\big\|
    +
    \frac{\beta^2}{\sqrt n} \big\|\hm^{k-1}-\hm^k\big\| \\
    &~~~+ 
    \beta^2 (q_*(\beta,t)-q_k(\beta,t))
    +
    o_{n,\mathbb P}(1).
\end{align*}
Here $o_{n,\mathbb P}(1)$ denotes terms which converge to $0$ in probability as $n\to\infty$. 
By \eqref{eq:z-converge}, \eqref{eq:qk*-lim} and the bound $(q_k(\beta,t)/t)\in [c,C]$ 
\[
    \lim_{k\to\infty}\sup_{t\in (0,T)}\plim_{n\to\infty}\frac{\|\bz^{k+1}-\bz^{k}\|}{\sqrt{tn}}=0.
\]
Moreover, $\|\hm^{k-1}-\hm^k\|\leq \|\bz^{k-1}-\bz^k\|$ since the function $x\mapsto \tanh(x)$ 
is $1$-Lipschitz. Finally \eqref{eq:uniform-gamma-limit} and \eqref{eq:gamma*-near0} of Lemma~\ref{lem:properties} imply
\[
    \lim_{k\to\infty}\sup_{t\in (0,\T]} \frac{q_*(\beta,t)-q_k(\beta,t)}{\sqrt{t}}=0 \, .
\]
Combining the above statements concludes the proof.
\end{proof}

We next obtain control on the Hessian $\nabla^2\cuF_{\sTAP}(\;\cdot\;;\by,q)$. 
As anticipated in Remark \ref{rmk:Beta}, this  is the only part of our proof that requires $\beta<1/2$
 instead of $\beta<1$.
\begin{lemma}
\label{lem:local-convex}
Let $\beta >0$, $\by \in \R^n$ and $q \in [0,1]$. Then for all $\m\in (-1,1)^n$,
 \begin{equation}
  \label{eq:global-smooth}
     \big(1-\beta\|\bA\|_{\textup{op}}\big) \D(\m) \preceq  \nabla^2\cuF_{\sTAP}(\m;\by,q)\preceq \big(1+\beta^2 + \beta\|\bA\|_{\textup{op}} \big) \D(\m) \, .
 \end{equation}
 In particular if $\beta\le \frac{1}{2}-c$, for $c>0$, then with probability $1-o_n(1)$, for all $\m\in (-1,1)^n$,
 \begin{equation}  \label{eq:local-convex}
c \, \D(\m) \preceq \nabla^2\cuF_{\sTAP}(\m;\by,q) \preceq 2 \D(\m) \, .
 \end{equation}
\end{lemma}

\begin{proof}
The upper and lower bounds in Eq.~\eqref{eq:global-smooth} are obtained from~\eqref{eq:hess-F} using the fact that $\D(\m)\succeq \bI_n$ for all $\m\in (-1,1)^n$.  
%$\nabla^2\cuF_{\sTAP}(\m;\by;q)=-\beta \bA+\D(\m)+ \beta^2\left(1-q\right)\bI_n$, 
Further, we use the fact that $\|\bA\|_{\text{op}}\leq 2+o_n(1)$ with probability $1-o_n(1)$. Therefore, Eq.~\eqref{eq:local-convex} follows from the assumption $\beta\le \frac{1}{2}-c$.
\end{proof}

As mentioned above, our convergence analysis of NGD, and proof of Lemma~\ref{lem:local-landscape}
are given in Appendix~\ref{app:NGD}. 
The key insight is that the main iterative step in line \ref{line:explicit-NGD-step} of 
Algorithm~\ref{alg:NGD} can be expressed as a version of mirror descent. 
Define the concave function $h(\m)=\sum_{i=1}^n h(m_i)$ for $\m\in (-1,1)^n$
(recall that $h(x) := -((1+x)/2)\log((1+x)/2)-((1-x)/2)\log((1-x)/2)$). 
Following \cite{lu2018relatively}, we define for $\m,\n\in (-1,1)^n$ the Bregman divergence
\begin{equation}
\begin{aligned}
\label{eq:bregman}
    \Dh(\m,\n)&=-\bh(\m)+\bh(\n)+\langle \nabla \bh(\n),\m-\n\rangle\ .
\end{aligned}
\end{equation}
Then with $L=1/\eta$, the update in line~\ref{line:explicit-NGD-step} admits the alternate description
\begin{equation}
\label{eq:NGD}
    \hm^{+,k+1}=\argmin_{\bx\in (-1,1)^n} \big\langle \nabla \cuF_{\sTAP}(\hm^{+,k};\by,q),\bx-\hm^{+,k} \big\rangle + L\cdot \Dh(\bx,\hm^{+,k})\, .
\end{equation} 
We will use this description to prove convergence.

\begin{remark}
If the Hessian $\nabla^2 \cuF_{\sTAP}$ were bounded above and below by constant multiples of the
 \textbf{identity} matrix instead of $\D(\m)$, then we could use simple gradient descent instead
 of NGD in Algorithm \ref{alg:Mean}. This would also simplify the proof.
 However, $\nabla^2 \cuF_{\sTAP}$ is not bounded above near the boundaries of $(-1,+1)^n$.
 The use of NGD to minimize TAP free energy was introduced in \cite{celentano2021local},
 which however considered a different regime in the planted model. 
\end{remark}

\begin{remark}\label{rmk:Beta2}
Our proof of Lemma~\ref{lem:local-landscape} does not require $\nabla^2\cuF_{\sTAP}$ to be globally convex.
Instead, we only use the fact that, with probability $1-o_n(1)$,
\[
    \nabla^2\cuF_{\sTAP}(\m;\by,q)\succeq c\D(\m),\quad\quad \forall\m\in B\left(\hm^{\sAMP},\sqrt{\eps tn}\right)\cap (-1,1).
\]
For $\beta\in [1/2,1)$ we expect only this weaker guarantee to hold.
 We believe the technique of \cite{celentano2021local} could be used
 to prove such local strong convexity in the full regime $\beta\in [0,1)$.
\end{remark}

%%%%%%%%%%%%%%%
\subsection{Continuous limit and proof of Theorem~\ref{thm:main}}
\label{sec:proofMain}
We fix $(\beta,\T)$ and choose constants $K_{\sAMP}=K_{\sAMP}(\beta,\T,\eps)$,
 $\rho_0=\rho_0(\beta,\T, \eps, K_{\sAMP})$, $\rho\in (0,\rho_0)$
 and $K_{\sNGD} = K_{\sNGD}(\beta,\T,\eps,\rho)$ so that Lemma~\ref{lem:local-landscape} holds. 
 
 %We assume $\T\in\delta\Z$ and consider the times $(\delta,2\delta,\cdots,\T)$. 
 We couple the discretized process $(\hby_{\ell})_{\ell\ge 0}$ defined in Eq.~\eqref{eq:approx} 
(line \ref{step:DiscreteSDE} of Algorithm \ref{alg:Sampling})
 to the continuous time process $(\by(t))_{t\in \R_{\geq 0}}$ (cf. Eq.~\eqref{eq:Q})
 via the driving noise, as follows: 
\begin{equation}
\label{eq:BM-couple}
    \bw_{\ell+1}=\frac{1}{\sqrt{\delta}} \int_{\ell\delta}^{(\ell+1)\delta}\rmd \bB(t) \, .
\end{equation}
We denote by $\hm(\bA,\by)$ the output of the mean estimation algorithm 
\ref{alg:Mean} on input $\bA,\by$.  By Lemma~\ref{lem:local-landscape}, which ensures that,
 for any $t\in (0,\T]$, 
with probability $1-o_n(1)$,
\begin{equation}
\label{eq:TAP-near-optimizer}
    \big\|\hm(\bA,\by(t))-\m_*(\bA,\by(t);q_*(\beta,t))\big\|\leq 
     \rho \sqrt{t n} \, .
\end{equation}
Here and below we note explicitly the dependence of $\m_*$ on $t$ via $q_*$.
The next lemma provides a crude estimate on the Lipschitz continuity of AMP with respect to its 
input.

\begin{lemma}
\label{lem:AMP-lip}
Recall that $\AMP(\bA,\by;k)\in \R^n$ denotes the output of the AMP 
algorithm on input $(\bA,\by)$, after $k$ iterations, cf. Eq.~\eqref{eq:AMP}.
If $\|\bA\|_{\textup{op}}\leq 3$, then, for any $\by,\hby\in \R^n$, 
\begin{equation}
\label{eq:AMP-lip}
    \big\|\atanh\big(\AMP(\bA,\by;k)\big) - \atanh\big(\AMP(\bA,\hby;k)\big)\big\|_2\leq k6^k \,\|\by-\hby\|_2\, .
\end{equation}
\end{lemma}

\begin{proof}
For $0\leq j\leq k$, set:
\begin{align*}
 \m^j&=\AMP(\bA,\by;j),
 &&\bz^j=\atanh(\m^j),
 &&\sb_j=\frac{\beta^2}{n}\sum_{i=1}^n \big(1-\tanh^2(z^j_i)\big),
 \\
\hm^j&=\AMP(\bA,\hby;j),
&&\hbz^j=\atanh(\hm^j),
&&\hat \sb_j=\frac{\beta^2}{n}\sum_{i=1}^n \big(1-\tanh^2(\hat z^j_i)\big).
\end{align*}
Using the AMP update equation (line \ref{eq:AMP-main-step} of Algorithm \ref{alg:Mean})
 and the fact that $\tanh(\, \cdot\, )$ is $1$-Lipschitz, we obtain
\begin{align*}
    \|\bz^{j+1}-\hbz^{j+1}\|
    &\leq \|\beta \bA(\m^j-\hm^j)\| + 
    \|\by-\hby\| + 
    \|\sb_j \m^{j-1}-\sb_j \hm^{j-1}\| + 
    \|\sb_j \hm^{j-1}-\hat \sb_j \hm^{j-1}\| 
    \\
     &\leq 3\beta \|\bz^{j}-\hbz^{j}\| + \|\by-\hby\| + \sb_j \|\bz^{j-1}-\hbz^{j-1}\| + |\sb_j-\hat \sb_j|\sqrt{n} \, .
\end{align*}
Note that $|1-\tanh^2(x)|\leq 1$ for all $x\in\mathbb R$ and $|b_j|\leq \beta^2$. Setting $E_j=\max_{i\leq j}\|\bz^{i+1}-\hbz^{i+1}\|$, we find
\begin{align*}
    E_{j+1} &\le (3\beta^2+3\beta)E_j + \|\by-\hby\| \\
    &\le 6 E_j + \|\by-\hby\| \, .
\end{align*}
It follows by induction that
\[
    E_j \leq j 6^j  \|\by-\hby\| \, .
\]
Setting $j=k$ concludes the proof.
\end{proof}

Define the random approximation errors 
\begin{align}
    A_{\ell} &:= \frac{1}{\sqrt{n}} \big\|\hby_{\ell}-\by({\ell}\delta)\big\| \, ,\\
    B_{\ell} &:= \frac{1}{\sqrt{n}} \big\|\hm(\bA,\hby_{\ell})-\m(\bA,\by({\ell}\delta))\big\| \, .
\end{align}
Note that $A_0=B_0=0$. In the next lemma we bound the above quantities:

\begin{lemma}
\label{lem:discretization-works}
For $\beta<1/2$ and $\T>0$, there exists a constant $C = C(\beta) <\infty$,
and a deterministic non-negative sequence $\xi(n)$ with $\lim_{n\to\infty}\xi(n)= 0$  
such that the following holds with probability $1-o_n(1)$. 
For every $\ell \geq 0$, $\delta \in (0,1)$ such that $\ell\delta \le \T$,
\begin{align}
    \label{eq:A-bound}
    A_{\ell}\leq Ce^{C\ell\delta}\ell\delta \big(\rho\sqrt{\ell\delta}+\sqrt{\delta}\big) + \xi(n) \, ,\\
    \label{eq:B-bound}
     B_{\ell}\leq Ce^{C\ell\delta}\ell\delta \big(\rho\sqrt{\ell\delta}+\sqrt{\delta}\big)+ C \rho\sqrt{\ell\delta}+ \xi(n) \, .
\end{align}

\end{lemma}

\begin{proof}
Throughout the proof, we denote by $\xi(n)$ a deterministic 
non-negative sequence $\xi(n)$ with $\lim_{n\to\infty}\xi(n)= 0$, which can change from 
line to line. Also, $C$ will denote a generic constant that may depend on $\beta,\T,K_{\sAMP}$.

The proof proceeds by induction on $\ell$. As the base case is trivial,
we assume the result holds for all $ j \le \ell$ and we prove it for $\ell+1$.
 We first claim that with probability $1-o_n(1)$,
\begin{equation}
\label{eq:A-recursion}
    A_{{\ell}+1}\leq A_{\ell}+\delta B_{\ell}+C\delta^{3/2}.
\end{equation}
Indeed, using \eqref{eq:BM-couple} we find
\begin{align*}
    A_{{\ell}+1}-A_{\ell}&\leq n^{-1/2}\int_{{\ell}\delta}^{({\ell}+1)\delta} \big\|\hm(\bA,\hby_{\ell})-\m(\bA,\by(t)) \big\|\, \rmd t
    \\
    &\leq 
    \delta n^{-1/2}\Big(\big\|\hm(\bA,\hby_{\ell})-\m(\bA,\by({\ell}\delta))\big\| 
    +
     \sup_{t\in [{\ell}\delta,({\ell}+1)\delta]} \big\|\m(\bA,\by(t))-\m(\bA,\by({\ell}\delta))\big\|\Big)
    \\
    &\leq 
    \delta B_{\ell} 
    + 
    \delta n^{-1/2}\cdot\sup_{t\in [{\ell}\delta,({\ell}+1)\delta]} \big\|\m(\bA,\by(t))-\m(\bA,\by({\ell}\delta))\big\| 
    \\
    &\leq 
    \delta B_{\ell} + C(\beta)\delta^{3/2} + \xi(n)\, ,
\end{align*}
where the last line holds with high probability by Lemma~\ref{lem:uniform-path} and
 Eq.~\eqref{eq:gamma*-Lipschitz} of Lemma~\ref{lem:properties}.
 Using this bound together with the inductive hypothesis on
 $A_{\ell}$ and $B_{\ell}$, we obtain
\begin{align*}
    A_{{\ell}+1}&\leq Ce^{C({\ell}+1)\delta}\ell\delta (\rho\sqrt{\ell\delta} + \sqrt{\delta})+ C\rho\delta\sqrt{\ell\delta} + C\delta^{3/2} + \xi(n)\\
    &\leq Ce^{C(\ell+1)\delta}(\ell+1)\delta (\rho +\sqrt{\delta}) + \xi(n) \, .
\end{align*}
This implies Eq.~\eqref{eq:A-bound} for ${\ell}+1$.

We next show that Eq.~\eqref{eq:B-bound} holds with $\ell$ replaced by  $\ell+1$.
% we first use Lemma~\ref{lem:local-landscape}. 
By the bound \eqref{eq:A-bound} for ${\ell}+1$,
taking $\delta \le \delta(\beta,\eps,K_{\sAMP},\T)$ and $\rho\in (0,\rho_0)$
 $\rho=  \rho(\beta,\eps,K_{\sAMP},\T)$ ensures that
\[
    A_{\ell+1}\leq \frac{c\sqrt{\eps \ell \delta}}{200 K_{\sAMP}6^{K_{\sAMP}}} \, ,
\]
where $\eps$ can be chosen an arbitrarily small constant.
So by Lemma~\ref{lem:AMP-lip}, we have with probability $1-o_n(1)$,
\begin{align*}
    \big\|\atanh(\AMP(\bA,\by((\ell+1)\delta);K_{\sAMP}))-\atanh(\AMP(\bA,\hby_{\ell+1};K_{\sAMP})) \big\|_2
    &\leq 
    K_{\sAMP} 6^{K_{\sAMP}}  A_{\ell+1} \sqrt{n}\\
    &\leq \frac{c\sqrt{\eps \ell \delta n}}{200}\, .
\end{align*}
By choosing $\eps\le \eps_0(\beta,\T)$, we obtain that Lemma~\ref{lem:local-landscape}, 
part \ref{it:landscape-NGD} applies. We thus find
\[
    \|\hm(\bA,\hby_{\ell+1})-\m_*(\bA,\hby_{\ell+1})\|\leq \rho\sqrt{\ell\delta n}\, .
\]
Using parts 3 and 2 respectively of Lemma~\ref{lem:local-landscape} on the other terms below, 
by triangle inequality  we obtain
(writing for simplicity $q_{\ell}:=q_{*}(\beta,\ell\delta)$)
\begin{equation}
\begin{aligned}
\label{eq:low-error-TAP}
    \|\hm(\bA,\hby_{{\ell}+1})-\m(\bA,\by(({\ell}+1)\delta))\|
    &\leq 
    \|\hm(\bA,\hby_{\ell+1})-\m_*(\bA,\hby_{\ell+1};q_{\ell+1})\| \\
    &\quad\quad+ \|\m_*(\bA,\hby_{\ell+1};q_{\ell+1})-\m_*(\bA,\by((\ell+1)\delta);q_{\ell+1})\| \\
    &\quad\quad + \|\m_*(\bA,\by((\ell+1)\delta);q_{\ell+1})-\m(\bA,\by(({\ell}+1)\delta))\|\\
    &\leq 
    \big(\rho\sqrt{\ell\delta} + c^{-1}A_{{\ell}+1}+\rho\sqrt{\ell\delta} + \xi(n)\big)\sqrt{n} \, .
\end{aligned}
\end{equation}
% We next write
% \begin{align*}
%     \|\hat{\m}_{\NGD}(\bA,\hby_{{\ell}+1})-\m(\bA,\by(({\ell}+1)\delta))\|&\leq \|\hat{\m}_{\NGD}(\bA,\hby_{{\ell}+1})-\hat{\m}_{\NGD}(\bA,\by(({\ell}+1)\delta))\|]\\
%     &\quad +\|\hat{\m}_{\NGD}(\bA,\by(({\ell}+1)\delta))-\m(\bA,\by(({\ell}+1)\delta))\|.
% \end{align*}
% The first term of the right-hand side is bounded by \eqref{eq:low-error-TAP}. Summing the approximation errors in Lemma~\ref{lem:local-landscape} (the middle of which is $0$) shows that the last term is at most $\frac{\delta^{3/2}\sqrt{n}}{2}$. 
In other words with probability $1-o_n(1)$,
\[
    B_{{\ell}+1}\leq c^{-1}A_{{\ell}+1} +2\rho\sqrt{\ell\delta} + \xi(n) \, .
\]
Using this together with the bound \eqref{eq:A-bound} for $\ell+1$
  verifies the inductive step for \eqref{eq:B-bound} and concludes the proof. 
% \mscomment{Remainder of proof: if small error so far, then $B_{\ell}\leq CA_{\ell}$ thanks to local Lipschitz result. Then, verify that the induction fits together and everything works.}
\end{proof}

Finally we show that standard randomized  rounding is continuous in $W_{2,n}$. 
\begin{lemma}
\label{lem:rounding-safe}
Suppose probability distributions $\mu_1,\mu_2$ on $[-1,1]^n$ are given. 
Sample $\m_1\sim \mu_1$ and $\m_2\sim\mu_2$ and let $\bx_1,\bx_2\in \{-1,+1\}^n$ be standard 
randomized roundings, respectively of $\m_1$ and $\m_2$. (Namely,
the coordinates of $\bx_i$ are conditionally independent given $\m_i$,
with $\E[\bx_i|\m_i]=\m_i$.) Then
\[
    W_{2,n}(\mathcal L(\bx_1),\mathcal L(\bx_2))\leq 2\sqrt{W_{2,n}(\mu_1,\mu_2)}\, .
\] 
\end{lemma}

\begin{proof}
Let $(\m_1,\m_2)$ be distributed according to a $W_{2,n}$-optimal coupling
 between $\mu_1,\mu_2$. Couple the roundings $\bx_1,\bx_2$ by choosing i.i.d.\ uniform 
 random variables $u_i\sim\Unif ([0,1])$ for $i\in [n]$, and for $(i,j)\in [n]\times \{1,2\}$ setting
\begin{align*}
 (\bx_j)_i &= 
\begin{cases}
 +1, & \mbox{if}~ u\leq \frac{1+(\m_j)_i}{2} \, ,\\
-1, & \mbox{else.} 
    \end{cases}
\end{align*}
Then it is not difficult to see that
\begin{align*}
    \frac{1}{n} \E\big[\|\bx_1-\bx_2\|^2~|(\m_1,\m_2)\big] &= \frac{2}{n}\sum_{i=1}^n |(\m_1)_i-(\m_2)_i|\\
    & \leq 2\sqrt{\frac{1}{n}\|\m_1-\m_2\|^2 }.
\end{align*}
Averaging over the choice of $(\m_1,\m_2)$ implies the result.
\end{proof}

\begin{proof}[Proof of Theorem~\ref{thm:main}]
Set $\ell=L=\T/\delta$ and $\rho=\sqrt{\delta}$ in Eq.~\eqref{eq:B-bound}. With all laws 
$\cL(\, \cdot\,)$ conditional on $\bA$ below, we find
\begin{align*}
    \E W_{2,n}(\mu_{\bA},\cL(\hm(\bA,\hby_L)))
    &\leq 
    \E W_{2,n}(\mu_{\bA},\cL(\m(\bA,\by(\T))))
    +
    \E W_{2,n}(\cL(\m(\bA,\by(\T)))),\cL(\hm(\bA,\hby_L)))
    \\
    &\leq 
    \T^{-1/2}+C(\beta,\T)\sqrt{\delta} +o_n(1).
\end{align*}
Here the first term was bounded by Eq.~\eqref{eq:W2bound} in Section~\ref{sec:stochloc} and the second by Eq.~\eqref{eq:B-bound}. Taking $\T$ sufficiently large, $\delta$ sufficiently small, and $n$ sufficiently large, we may obtain 
\[
   \E  W_{2,n}\big(\mu_{\bA},\cL(\hat{\m}_{\NGD}(\bA,\hby_L))\big)\leq \frac{\eps^2}{4}
\]
for any desired $\eps>0$. Applying Lemma~\ref{lem:rounding-safe} shows that
\[
    \E W_{2,n}(\mu_{\bA},\bx^{\salg})\leq \eps \, .
\]
The Markov inequality now implies that \eqref{eq:main} holds with probability $1-o_n(1)$ as desired.
\end{proof}

%%%%%%%%%%%%%%%
\section{Algorithmic stability and disorder chaos}
\label{sec:stable}
In this section we
prove Theorem \ref{thm:stable} establishing that our sampling algorithm,
Algorithm \ref{alg:Sampling} is stable. Next, we prove that the Sherrington-Kirkpatrick 
measure $\mu_{\bA,\beta}$ exhibits $W_2$-disorder chaos for $\beta>1$, proving Theorem
\ref{thm:disorder_chaos_sk} and deduce that no stable algorithm 
can sample in normalized $W_2$ distance for $\beta>1$, see Theorem \ref{thm:disorder-stable-LB}.

\subsection{Algorithmic stability: Proof of Theorem \ref{thm:stable}}
\label{sec:algo_stable}
Recall Definition \ref{def:Stable}, defining  sampling algorithms
as measurable functions $\ALG_n:(\bA,\beta,\omega)\mapsto \ALG_n(\bA,\beta,\omega)\in [-1,1]^n$ 
where $\beta\geq 0$ and $\omega$ is an independent random variable taking values in some probability 
space. 

\begin{remark}
In light of Lemma~\ref{lem:rounding-safe}, we can always turn a stable sampling algorithm 
$\ALG$ with codomain $[-1,1]^n$ into a stable sampling algorithm with binary output:
\[
    \widetilde{\ALG}_n(\bA,\beta,\widetilde\omega) \in \{-1,+1\}^n\, .
\]  
Indeed this is achieved by standard randomized rounding, i.e., drawing a
(conditionally independent)  random binary value with mean 
$\big(\widetilde{\ALG}(\bA,\beta,\widetilde\omega)\big)_i$ for each 
coordinate $1 \le i \le n$.
\end{remark}

Recall the definition of the interpolating family $(\bA_s)_{s\in [0,1]}$ whereby 
 $\bA_{0},\bA_{1} \sim \GOE(n)$ i.i.d.\ and 
\begin{equation}
\label{eq:interpolation-path}
    \bA_{s} = \sqrt{1-s^2}\bA_{0} + s \bA_{1} \, ,\quad s\in [0,1] \, ,
\end{equation}
We take $\mu_{\bA_s,\beta}(\bx)\propto \exp \big\{  (\beta/2) \langle  \bx , \bA_s \,\bx \rangle  \big\}$
to be the corresponding Gibbs measure.  

We start with the following simple estimate.
\begin{lemma}
\label{lem:grad-stable}
There exists an absolute constant $C >0$ such that
\begin{equation}
\label{eq:standard-alg-stable}
    \inf_{s\in (0,1)}
    \mathbb P\Big(
    \big\| \bA_{0}\bu- \bA_{s}\bv \big\| \leq C(\|\bu-\bv\|+s\sqrt{n}) \, ,~~\forall~\bu,\bv\in [-1,1]^n \Big)=1-o_n(1) \, .
\end{equation}
\end{lemma}
\begin{proof}
We write 
\begin{align*} 
\big\| \bA_{0}\bu- \bA_{s}\bv \big\|  &\le \big\| \bA_{0}\bu- \bA_{0}\bv \big\|  + \big\| \bA_{0}\bv- \bA_{s}\bv \big\| \\
&\le \|\bA_0\|_{\text{op}} \, \big\|\bu - \bv\big\| + \big\| (1-\sqrt{1-s^2})\bA_0 - s \bA_1\big\|_{\text{op}} \, \|\bv\| \, .
\end{align*}
We note that $(1-\sqrt{1-s^2})\bA_0 - s \bA_1 \stackrel{\rmd}{=} \sqrt{2(1-\sqrt{1-s^2})} \bA_0$ and $\sqrt{2(1-\sqrt{1-s^2})} \sim s$ for small $s$ and this quantity is bounded above by a constant for any $s \in [0,1]$.
The result follows since $\|\bA_0\|_{\text{op}} \le 2.1$ with probability $1- o_n(1)$.
\end{proof}

%
%\begin{lemma}
%\label{lem:grad-stable}
%There exists a constant $C=C(\xi)$ such that
%\begin{equation}
%\label{eq:standard-alg-stable}
%    \inf_{s\in (0,1)}
%    \mathbb P\Big(
%    \big\|\nabla H_{n,0}(\bu)-\nabla H_{n,s}(\bv) \big\| \leq C(\|\bu-\bv\|+s\sqrt{n}) \, ,~~\forall~\bu,\bv\in [-1,1]^n \Big)=1-o_n(1) \, .
%\end{equation}
%\end{lemma}
%
%\begin{proof}
%
%We write
%\[
%    \big\|\nabla H_{n,0}(\bu)-\nabla H_{n,s}(\bv) \big\|= \big\|\nabla H_{n,0}(\bu)-\nabla H_{n,0}(\bv) \big\| + \big\|\nabla H_{n,0}(\bv)-\nabla H_{n,s}(\bv)\big\|
%\]
%and estimate both terms. For a suitable constant $C(\xi) >0$, with probability $1-o_n(1)$ we have
%\[
%    \big\|\nabla H_{n,0}(\bu)-\nabla H_{n,0}(\bv) \big\| \leq C\|\bu-\bv\| \, ,\quad \forall~\bu,\bv\in [-1,1]^n \, .
%\]
%See e.g., \cite[Proposition 1.1]{huang2021tight} for a proof of this fact. \aelcomment{Write proof of this.}
%By computing the variance of entries in $H_{n,0}-H_{n,s}$, it follows that 
%\[
%    \frac{H_{n,0}-H_{n,s}}{\sqrt{2-2\sqrt{1-s^2}}}
%\]
%has the same law as $H_n$. Moreover, $\sqrt{2-2\sqrt{1-s^2}} = \Theta(s)$ for small $s$. It again follows from \cite[Proposition 1.1, Inequality (1.11)]{huang2021tight} that with probability $1-o_n(1)$,
%\[
%    \big\|\nabla H_{n,0}(\bv)-\nabla H_{n,s}(\bv) \big\|\leq Cs\sqrt{n} \, , \quad \forall~\bv\in [-1,1]^n \, .
%\]
%Combining the two results concludes the proof.
%\end{proof}

\begin{proposition}
\label{prop:standard-alg-stable}
Suppose an algorithm $\ALG$ is given by an iterative procedure
\begin{align*}
    &\bz^{k+1}
    = 
    G_k\left((\bz^j,\beta\bA\m^j,\bA\m^j,\beta^2\m^j,\bw^j)_{0\leq j\leq k}\right),
    \quad 0\leq k\leq K-1,\\
    &\m^k =\rho_k(\bz^k),\quad 0\leq k\leq K-1,\\
    &\ALG_n(\bA,\beta,\omega) := \m^K \, ,
\end{align*}
where the sequence $\omega=(\bw^0,\dots,\bw^{K-1})\in (\mathbb R^n)^K$, the initialization 
$\bz^0\in\mathbb R^n$, and $\bA$ are mutually independent, and the functions 
$G_k:(\mathbb R^n)^{5k+5}\to \mathbb R^n$ and $\rho_k:\mathbb R^n\to [-1,1]^n$ are $L_0$-Lipschitz for 
$L_0\geq 0$ an $n$-independent constant. Then $\ALG$ is both disorder-stable and temperature-stable.
\end{proposition}

\begin{proof}

Let us generate iterates $\bz^k=\bz^{k}(\bA_0,\beta)$ and $\wtbz^k=\bz^{k}(\bA_s,\wtbeta)$ for $0\leq k\leq K$ using the same initialization $\bz^0=\wtbz^0$ and external randomness $\omega=(\bw^0,\dots,\bw^{K-1})$, but with different Hamiltonians and inverse temperatures. Similarly let $\m^k=\rho_k(\bz^k)$ and $\wtm^k=\rho_k(\wtbz^k)$. We will allow $C$ to vary from line to line in the proof below.

First by Lemma~\ref{lem:grad-stable}, with probability $1-o_n(1)$,
\begin{align*}
    \|\beta \bA_0 \m^k- \wtbeta\bA_s \wtm^k\|
    &\leq
    \|\beta \bA_0 \m^k - \beta \bA_s \wtm^k\|  
    +
    \|\beta \bA_s \wtm^k - \wtbeta \bA_s \wtm^k\|
    \\
    &\leq
    C\beta \|\m^k-\wtm^k\| + C\beta s\sqrt{n}
    + 
    |\beta-\wtbeta|\cdot \| \bA_s \wtm^k\|\\
    &\leq
    C(\|\m^k-\wtm^k\|+s\sqrt{n}+|\beta-\wtbeta|\sqrt{n})\, .
\end{align*}
 Similarly as long as $\wtbeta\leq 2\beta$ so that $|\beta^2-\wtbeta^2|\leq 3\beta|\beta-\wtbeta|$, we have
\begin{align*}
    \|\beta^2\m^k-\wtbeta^2\wtm^k\|
    &\leq
    \|\beta^2\m^k-\beta^2\wtm^k\|
    +
    \|\beta^2\wtm^k-\wtbeta^2\wtm^k\|\\
    &\leq
    \beta^2 \|\m^k-\wtm^k\| 
    +
    3\beta|\beta-\wtbeta|\sqrt{n}.
\end{align*}
It follows that the error sequence 
\[
    A_k=\frac{1}{\sqrt{n}}\max_{j\leq k}\|\bz^{j+1}(\bA_0,\beta)-\bz^{j+1}(\bA_s,\wtbeta)\|
\] 
satisfies with probability $1-o_n(1)$ the recursion
\begin{align*}
    A_{k+1} & \leq L_0k^{1/2}C(A_k + s + |\beta-\wtbeta|) \, ,\\
    A_0 &= 0 \, ,
\end{align*}
for a suitable $C=C(\beta)$. It follows that with probability $1-o_n(1)$,
\begin{equation}\label{eq:A_k}
    A_K\leq \sum_{k=1}^K (L_0k^{1/2}C)^k (s+ |\beta-\wtbeta|) \leq K(L_0KC)^K (s+|\beta-\wtbeta|) \, .
\end{equation}
Since $\|\m^K(\bA_0)-\m^K(\bA_s)\|\leq 2\sqrt{n}$ almost surely, we obtain for any $\eta>0$
\[
    n^{-1} \E\left[\big\|\m^K(\bA_0)-\m^K(\bA_s)\big\|^2\right] \leq \big(L_0K(L_0KC)^K (s+|\beta-\wtbeta|)\big)^2 + \eta
\]
if $n\geq n_0(\eta)$ is large enough so that Eq.~\eqref{eq:A_k} holds with probability at least $1-\frac{\eta}{4}$. The stability of the algorithm follows.   
\end{proof}

\begin{proof}[Proof of Theorem~\ref{thm:stable}]

We show that Algorithm \ref{alg:Sampling} with 
$n$-independent parameters $(\beta,\eta,K_{\sAMP},K_{\sNGD},L,\delta)$
 is of the form in Proposition~\ref{prop:standard-alg-stable} for a constant 
 $L_0=L_0(\beta,\eta,K_{\sAMP},K_{\sNGD},L,\delta)$.
 Indeed note that the algorithm goes through $L$ iterations, indexed by
 $\ell\in\{0,\dots,L-1\}$. 
 
 During each of these iterations, two loops 
 are run (here we modify the notation introduced in Algorithm \ref{alg:Mean} and Algorithm \ref{alg:Sampling},
 to account for the dependence on $\ell$, and to get closer to the notation of Proposition~\ref{prop:standard-alg-stable}):
\begin{enumerate}
    \item\label{it:stable-case-1}
    The AMP loop, whereby, for $k = 0,\cdots,K_{\sAMP}-1$,
    \begin{align}
    \hm^{\ell,k} &= \tanh(\bz^{\ell,k} ) , ~~~~~~~ \sb(\hm^{\ell,k})= \frac{\beta^2}{n}
    \sum_{i=1}^n\tanh'(z^{\ell,k}_i)\, ,\\
\bz^{\ell,k+1} &= \beta \bA \hm^{\ell,k} + \hby_{\ell} - \sb(\bz^{\ell,k})\,\hm^{\ell,k-1}\, .
    \end{align}
    (Here $\tanh'(x)$ denotes the first derivative of $\tanh(x)$.)
    \item\label{it:stable-case-2}
    The NGD loop, whereby,  for $k = K_{\sAMP},\cdots,K_{\sAMP}+K_{\sNGD}-1$,
    setting $q_{\ell} = q_{K_{\sAMP}}(\beta,t=\ell\delta)$,
    \begin{align}
    \hm^{\ell,k} & = \tanh(\bz^{\ell,k} ) \, ,\\
    \bz^{\ell,k+1} &= \bz^{\ell,k} + \eta
     \big[\beta \bA\hm^{\ell,k}+\by_{\ell}-\bz^{\ell,k}
    - \beta^2\left(1-q_{\ell}\right)\m^{\ell,k} \big]\, . 
    \end{align}
\end{enumerate}
Further, recalling line  \ref{step:DiscreteSDE} of Algorithm \ref{alg:Sampling}, $\hby_{\ell}$
 is updated via
\begin{align}
\hby_{\ell+1} = \hby_{\ell} + \hm^{\ell,K_{\sAMP}+K_{\sNGD}} \, \delta + \sqrt{\delta} \, \bw_{\ell+1}\, .
\label{eq:Yell}
\end{align}

These updates take the same form as in  Proposition~\ref{prop:standard-alg-stable},
with iterations indexed by $(\ell,k)$, $\omega =(\bw_{\ell})_{\ell\le L}$, 
$\rho_{\ell,k}(\bz)=\tanh(\bz)$ for all $\ell,k$,  and 
\begin{align}
 G_{\ell,k}
 \left((\bz^{\ell',j},\beta\bA\hm^{\ell',j},\bA\hm^{\ell',j},\beta^2\hm^{\ell',j},\bw_{\ell'})_{\ell',j}\right)
 &=\beta \bA \hm^{\ell,k} + \hby_{\ell} - \sb(\bz^{\ell,k}) \hm^{\ell,k-1}\, ,\;\;\;\;\;
 0\le k\le K_{\sAMP}-1\, ,\label{eq:G-AMP}\\
  G_{\ell,k}\left(
  (\bz^{\ell',j},\beta\bA\hm^{\ell',j},\bA\hm^{\ell',j},\beta^2\hm^{\ell',j},\bw_{\ell'})_{\ell',j}
 \right)&= \nonumber\\
  =\bz^{\ell,k} + \eta
     \big[\beta \bA\hm^{\ell,k}+\by_{\ell}-&\bz^{\ell,k}
    - \beta^2\left(1-q_{\ell}\right)\m^{\ell,k} \big]\, ,\;\;\;\;\;
K_{\sAMP} \le k\le K_{\sAMP}+K_{\sNGD}-1\, .\label{eq:G-NGD}
\end{align}
 Notice that these functions depend on previous iterates both explicitly, as 
 noted, and implicitly through $\hby_{\ell}$. By summing up Eq.~\eqref{eq:Yell},
 we obtain
\begin{align}
\hby_{\ell} = \sum_{j=0}^{\ell-1} \hm^{j,K_{\sAMP}+K_{\sNGD}} \, \delta + 
\sqrt{\delta} \sum_{j=1}^{\ell}  \bw_{\ell+1}\, ,
\end{align}
which is Lipschitz in the previous iterates  $(\m^{j,k})_{j\le \ell-1,k<K_{\sAMP}+K_{\sNGD}}$. 
Since both \eqref{eq:G-AMP} and \eqref{eq:G-NGD} depend linearly on $\hby_{\ell}$
(with $n$-independent coefficients), it is sufficient to consider 
the explicit dependence on previous iterates of $G_{\ell,k}$.
Namely, it is sufficient to control the Lipschitz modulus of the following functions
\begin{align}
 \tG_{\ell,k}\left(\bz^{\ell,k},\beta\bA\hm^{\ell,k},\hm^{\ell,k-1}\right)
 &=\beta \bA \hm^{\ell,k} - \sb(\bz^{\ell,k}) \hm^{\ell,k-1}
 \, ,\;\;\;\;\;
 k< K_{\sAMP}\label{eq:TG1}\\
  \tG_{\ell,k}\left(\bz^{\ell,k},\beta\bA\m^{\ell,k},\beta^2\m^{\ell,k}\right)
  &=\bz^{\ell,k} + \eta
     \big[\beta \bA\hm^{\ell,k}-\bz^{\ell,k}
    - \beta^2\left(1-q_{\ell}\right)\hm^{\ell,k} \big]  
     \, ,\;\;\;\;\;
 k> K_{\sAMP}\, .\label{eq:TG2}
\end{align}

Consider first Eq.~\eqref{eq:TG1}. Since $|\tanh''(x)|\leq 2$ for all $x\in \R$, it follows that
\[
    |\sb(\bz)-\sb(\wtbz)|\leq \frac{2\beta^2}{n} \sum_{i=1}^n  |z_i-\tilde{z}_i|\leq \frac{2\beta^2}{\sqrt{n}} \|\bz-\wtbz\|_2 .
\]
Therefore, that for any $(\bu,\bv,\beta,\wtbu,\wtbv,\wtbeta)$
(noting explicitly the dependence of $\sb$ upon $\beta$):
\begin{align*}
    \|\sb_{\beta}(\bu)\tanh(\bv)-\sb_{\wtbeta}(\wtbu)\tanh(\wtbv)\|
    &\leq
    \|\sb_{\beta}(\bu)\tanh(\bv)-\sb_{\beta}(\wtbu)\tanh(\bv)\|
    +
    \|\sb_{\beta}(\wtbu)\tanh(\bv)-\sb_{\wtbeta}(\wtbu)\tanh(\wtbv)\|
    \\
    &\leq \frac{2\beta^2}{\sqrt{n}} \|\bu-\wtbu\| \cdot \|\tanh(\bv)\| + \Big(\frac{1}{n} \sum_{i=1}^n \tanh'(\tilde{u}_i)\Big)  \|\beta^2\tanh(\bv)-\wtbeta^2\tanh(\wtbv)\| 
    \\
    &\leq
    2\beta^2\|\bu-\wtbu\| + \|\beta^2\tanh(\bv)-\wtbeta^2\tanh(\wtbv)\| .
\end{align*}
Using this bound implies that the function $\tG$ of Eq.~\eqref{eq:TG1}
satisfies the Lipschitz assumption of Proposition~\ref{prop:standard-alg-stable}. 

Consider next Eq.~\eqref{eq:TG2}. Since this function is linear in its arguments, 
with coefficients independent of $n$, it follows that it satisfies  Lipschitz assumption of 
Proposition~\ref{prop:standard-alg-stable}. This completes the proof.
\end{proof}

\subsection{Hardness for stable algorithms: Proof of Theorems \ref{thm:disorder_chaos_sk} and \ref{thm:disorder-stable-LB}}
\label{sec:disorder}

Before proving Theorem \ref{thm:disorder_chaos_sk} and Theorem \ref{thm:disorder-stable-LB} 
we recall a known result about disorder chaos, already stated in Eq.~\eqref{eq:FirstDisorderChaos}. 
 Draw $\bx^0\sim \mu_{\bA,\beta}$ independently of  
$\bx^s\sim \mu_{\bA_s,\beta}$, and denote by $\mu^{(0,s)}_{\bA,\beta}:=\mu_{\bA,\beta}\otimes
\mu_{\bA^s,\beta}$ their joint distribution.
Then \cite[Theorem 1.11]{chatterjee2014superconcentration} implies that, for all $\beta\in (0,\infty)$,
\begin{align}\label{eq:FirstDisorderChaos-B}
\lim_{s\to 0}\lim_{n\to\infty}\E\mu^{(0,s)}_{\bA,\beta} \Big\{\Big(\frac{1}{n}\<\bx^0,\bx^s\>\Big)^2\Big\}= 0
\, .´
\end{align}

The following simple estimate will be used in our proof.
\begin{lemma}\label{lem:continuity}
Recall that $\cuP(\{-1,+1\}^n)$ denotes the space of probability distributions over 
$\{-1,+1\}^n$, and 
let the function $f : \cuP(\{-1,+1\}^n)^2 \to \R$ be defined as
\begin{equation}
f(\mu,\mu')=  \E_{(\bx,\bx')\sim \mu\otimes\mu'} \Big\{\frac{1}{n}|\langle\bx,\bx'\rangle| \Big\}\, .
\end{equation}
Then, for all $\mu_1,\mu_2, \nu_1,\nu_2 \in \cuP(\{-1,+1\}^n)$, we have 
\[\big| f(\mu_1,\nu_1) - f(\mu_2,\nu_2) \big| \le W_{2,n}(\mu_1,\mu_2) + W_{2,n}(\nu_1,\nu_2) \, .\]
\end{lemma} 
\begin{proof}
Let the vector pairs $(\bx^{\mu_1},\bx^{\mu_2})$ and $(\bx^{\nu_1},\bx^{\nu_2})$ be 
independently drawn from the optimal $W_{2,n}$-couplings of the pairs $(\mu_1,\mu_2)$ 
and $(\nu_1,\nu_2)$, respectively. Then we have:
\begin{align*}
 \Big| \E\big\{|\< \bx^{\mu_1} , \bx^{\nu_1}\>|\big\} -  
 \E\big\{|\< \bx^{\mu_2} , \bx^{\nu_2}\>|\big\} \Big| 
 &\le 
 \Big| \E\big\{|\< \bx^{\mu_1} , \bx^{\nu_1}\>| -  
 |\< \bx^{\mu_2} , \bx^{\nu_1}\>|\big\} \Big| +
 \Big| \E\big\{|\< \bx^{\mu_2} , \bx^{\nu_1}\>| -  
 |\< \bx^{\mu_2} , \bx^{\nu_2}\>|\big\} \Big| \\
&\le \sqrt{n} \Big(\E\big\|\bx^{\mu_1} -  \bx^{\mu_2}\big\| +  \E\big\|\bx^{\nu_1} - \bx^{\nu_2}\big\|\Big) \\
&\le \sqrt{n} \Big(\E\Big[\big\|\bx^{\mu_1} -  \bx^{\mu_2}\big\|^2\Big]^{1/2} +  \E\Big[\big\|\bx^{\nu_1} - \bx^{\nu_2}\big\|^2\Big]^{1/2}\Big)  \, ,
\end{align*}
where the second inequality follows from the fact that $\bx \mapsto |\<\bv,\bx\>|$ is Lipschitz
continuous with Lipschitz constant $\|\bv\|_2$. 
\end{proof}

We are now in position to prove Theorem~\ref{thm:disorder_chaos_sk}.
\begin{proof}[Proof of Theorem~\ref{thm:disorder_chaos_sk}]
Using the notations of the last lemma Eq.~\eqref{eq:FirstDisorderChaos-B}
 implies that for all $s \in (0,1]$,
\begin{equation}\label{eq:mu0mus} 
\lim_{n \to \infty} \E f( \mu_{\bA_s,\beta}, \mu_{\bA_0,\beta}) = 0 \, . 
\end{equation}
Therefore, Theorem~\ref{thm:disorder_chaos_sk} follows from Lemma \ref{lem:continuity} if we can 
show that 
$ f( \mu_{\bA_0,\beta}, \mu_{\bA_0,\beta})$ remains bounded away from zero.
 This is in turn a well-known consequence of the Parisi formula, as we recall below. 
 
 Define the free energy density of the SK model as
\begin{equation} \label{eq:free_energy}
F_n(\beta) = \frac{1}{n} \E\, \log \Big\{\sum_{\bx \in \{-1,+1\}^n} e^{\beta\langle \bx , \bA \bx \rangle/2}
\Big\} \, .
\end{equation}
The free energy $F_n$ is convex in $\beta$ and one obtains by Gaussian integration parts that
\begin{equation} \label{eq:derivative_free_energy}
\frac{\rmd ~}{\rmd \beta} F_n(\beta)=   \frac{\beta}{2} \Big( 1 - 
\E \mu_{\bA_0,\beta}\otimes\mu_{\bA_0,\beta}\Big\{\Big(\frac{1}{n}\<\bx_1, \bx_2\>\Big)^2\Big\}\Big) \, .
\end{equation}  
Moreover, the limit of $F_n(\beta)$ for large $n$ is known to exist for all $\beta>0$ 
and its value is given by the Parisi formula \cite{talagrand2006parisi}:
\begin{equation} \label{eq:parisi}
\lim_{n\to\infty} F_n(\beta) = \inf_{\zeta \in \cuP([0,1])} \Par_{\beta}(\zeta) \, ,
\end{equation}
where $\cuP([0,1])$ denotes the set of Borel probability measures supported on $[0,1]$, and
 $\Par_{\beta}$ is the Parisi functional at inverse temperature $\beta$; see for 
 instance~\cite{talagrand2006parisi} or~\cite[Chapter 3]{panchenko2013sherrington} for 
  definitions.   

The following properties are known:
\begin{enumerate}
    \item A unique minimizer $\zeta_{\beta}^*\in\cuP([0,1])$ of $\Par_{\beta}$ exists for all $\beta$~\cite{auffinger2015parisi}.
    \item \label{it:lowtemp}If $\beta>1$, then $\zeta_{\beta}^*$ is not an atom on $0$: $\zeta_{\beta}^* \neq \delta_0$. 
    This follows from Toninelli's theorem~\cite{toninelli2002almeida} that $\lim\sup_{n\to\infty}F_n(\beta) \le \log 2 + \beta^2/4-\eps(\beta)$
    for some continuous $\eps(\beta)$, with $\eps(\beta)>0$ when $\beta>1$. 
    \item The function $\beta \mapsto \Par_{\beta}(\zeta_{\beta}^*)$ is convex and differentiable at all $\beta > 0$, and
    \begin{equation} \label{eq:derivative_parisi}
    \frac{\rmd ~}{\rmd \beta} \Par_{\beta}(\zeta_{\beta}^*) = \frac{\beta}{2} \Big( 1 - \int q^2  \zeta_{\beta}^*(\rmd q) \Big) \, .
    \end{equation} 
    See for instance~\cite[Theorem 3.7]{panchenko2013sherrington} or~\cite[Theorem 1.2]{talagrand2006parisi-b} for a proof. 
%    \item\label{it:zeta-continuous} The map $\beta\mapsto\zeta^{\beta}_*$ is continuous, where the codomain is equiped with the uniform norm \cite[Corollary 4.2]{panchenkoextra}.
%\mscomment{Not sure if there is a real reference for this.}
\end{enumerate}

The convexity of $F_n$ implies that for almost all $\beta>0$, 
$\lim_{n\to\infty}F_n'(\beta) = \frac{\rmd ~}{\rmd \beta} \Par_{\beta}(\zeta_{\beta}^*) $. 
Using Eq.~\eqref{eq:derivative_free_energy} and  Eq.~\eqref{eq:derivative_parisi} we obtain 
\begin{equation}\label{eq:lim_f}
\lim_{n \to \infty} \frac{\beta}{2} \Big( 1 - 
\E \mu_{\bA_0,\beta}\otimes\mu_{\bA_0,\beta}\Big\{\Big(\frac{1}{n}\<\bx_1, \bx_2\>\Big)^2\Big)\Big\}
 = \frac{\beta}{2} \Big( 1 - \int q^2  \zeta_{\beta}^*(\rmd q) \Big)< \frac{\beta}{2}-\eps(\beta)\, ,
\end{equation}
where the last inequality holds for almost all $\beta>1$ by Property~\ref{it:lowtemp}
above. Since the both sides are non-decreasing and the right hand side is 
continuous, the inequality holds for all $\beta$. This is equivalent to
\begin{equation}\label{eq:mu0mu0}
\lim_{n \to \infty} \E f(\mu_{\bA_0,\beta}, \mu_{\bA_0,\beta}) >0 \, . 
\end{equation}
Now, using Eq.~\eqref{eq:mu0mus} and Eq.~\eqref{eq:mu0mu0}, together with the continuity of $f$ 
(Lemma~\ref{lem:continuity}) implies the claim of the theorem.
\end{proof}

We next prove that  Theorem~\ref{thm:disorder-stable-LB} is an immediate consequence
of Theorem~\ref{thm:disorder_chaos_sk}.  
\begin{proof}[Proof of Theorem \ref{thm:disorder-stable-LB}]
Fix $s\in (0,1)$ and $\mu^{\salg}_{\bA_s,\beta}$ be the law of  $\ALG_n(\bA_s,\beta,\omega)$ 
conditional on $\bA_s$. By the triangle inequality,
\begin{align*}
W_{2,n}(\mu_{\bA_s,\beta,s}, \mu_{\bA_0,\beta}) 
\le W_{2,n}(\mu_{\bA_s,\beta}, \mu^{\salg}_{\bA_s,\beta}) + 
W_{2,n}(\mu^{\salg}_{\bA_s,\beta}, \mu^{\salg}_{\bA_0,\beta}) + 
W_{2,n}(\mu^{\salg}_{\bA_0,\beta}, \mu_{\bA_0,\beta,0}) \, .
\end{align*}
 Taking expectations over $\bA$ and $\bA_s$, we have 
 $\E\big[W_{2,n}(\mu_{\bA_s,\beta}, \mu^{\salg}_{\bA_s,\beta})\big] = \E\big[W_{2,n}(\mu^{\salg}_{\bA_0,\beta}, 
 \mu_{\bA_0,\beta})\big]$. Further, by stability of the algorithm, 
 $\E\big[W_{2,n}(\mu^{\salg}_{\bA_s,\beta}, \mu^{\salg}_{\bA_0,\beta})\big] \to 0$ when
  $n \to \infty$ followed by $s \to 0$. Therefore, using Theorem~\ref{thm:disorder_chaos_sk} 
   and choosing $s$ sufficiently small, we obtain
 \[\liminf_{n\to\infty} \E \big[W_{2,n}(\mu^{\salg}_{\bA_0,\beta},~\mu_{\bA_0,\beta}) \big] \ge W_*>0\, .\]
 \end{proof}

%\vspace{5mm}
%\textbf{Acknowledgments.}
%\vspace{5mm}

% \newpage
\bibliographystyle{amsalpha}
% \bibliography{all-bibliography}
\bibliography{all-bibliography}

\newpage

\appendix

\section{Convergence analysis of Natural Gradient Descent}

\label{app:NGD}

The main objective of this appendix is to prove Lemma \ref{lem:local-landscape},
which we will do in Section \ref{app:local-landscape}, after some technical preparations in Section \ref{app:Preliminaries}.

\subsection{Technical preliminaries}
\label{app:Preliminaries}

\begin{definition}
\label{defn:c-strongly-convex}
Let $Q\subseteq (-1,1)^n $ be a convex set. We say that a twice differentiable function 
$\F:Q\to \mathbb R$ is relatively 
$c$-strongly convex if it satisfies 
\begin{equation}
\label{eq:c-strongly-convex}
    \nabla^2 \F(\m) \succeq  c \Dm \;\;\;\forall\;\m\in Q \, .
\end{equation}
We say it is relatively $C$-smooth if 
it satisfies 
\begin{equation}
\label{eq:C-smooth}
    \nabla^2 \F(\m) \preceq   C\Dm \;\;\;\forall\;\m\in Q \, .
\end{equation}
\end{definition}

As $\Dm=\nabla^2 (-\bh(\m))\succeq \bI_n$, it follows that \eqref{eq:c-strongly-convex} 
implies ordinary $c$-strong convexity in Euclidean norm. The next proposition connects 
relative strong convexity with the Bregman divergence introduced in Eq.~\ref{eq:bregman}.
\begin{proposition}[Proposition 1.1 in \cite{lu2018relatively}]
A twice differentiable function 
$\F:Q\to \mathbb R$ is relatively 
$c$-strongly convex if  and only if
\begin{equation}
\label{eq:alt-strongly-convex}
    \F(\m)\geq \F(\n)+\langle \nabla\F(\n),\m-\n\rangle + c \Dh(\m,\n),\quad\quad \forall\m,\n\in Q \, .
\end{equation}

\end{proposition}

\begin{lemma}
\label{lem:Bregman}
For $\m,\n\in (-1,1)^n$, 
\begin{align}
    \label{eq:Bregman-bigger}
    \Dh(\m,\n) &\geq \frac{\|\m-\n\|_2^2}{2} \, ,\\
    \label{eq:Bregman-bound}
    \Dh(\m,\n)&\leq 10 n\left(1+\frac{\|\atanh(\n)\|_2}{\sqrt{n}}\right)\, ,\\
\label{eq:Bregman-Lip}
    \Dh(\m,\n)&\leq \|\atanh(\m)-\atanh(\n)\|_2^2 \, .
\end{align}

\end{lemma}

\begin{proof}

Observe that $h''(x)=-1/(1-x^2)\leq -1$ for all $x\in (-1,1)$ with equality if and only if $x=0$. 
Therefore
\begin{align*}
    \Dh(\m,\n)&=
    \sum_{i=1}^n \int_{m_i}^{n_i}  (x-m_i) (-h''(x))\, \rmd x \\
    & = \sum_{i=1}^n \frac{(n_i-m_i)^2}{2} \, .
\end{align*}
This proves Eq.~\eqref{eq:Bregman-bigger}.

 Next, Eq.~\eqref{eq:Bregman-bound} follows from Eq.~\eqref{eq:bregman} and the fact that the binary 
 entropy $h:\mathbb R\to\mathbb R$ is uniformly bounded. 
 
 Finally Eq.~\eqref{eq:Bregman-Lip} follows from
\begin{align*}
    \Dh(\m,\n)&\le \<\nabla h(\n)-\nabla h(\m), \m-\n\>\\
    & = \big\langle \atanh(\m)-\atanh(\n),\m-\n\big\rangle \\
    & \leq \big\|\atanh(\m)-\atanh(\n)\big\|_2^2\, .
\end{align*}
Here in the last step we used that $\tanh(\cdot)$ is $1$-Lipschitz.
\end{proof}

\begin{lemma}
If $\F:Q\to \mathbb R$ is relatively $c$-strongly convex for some convex set $Q\subseteq (-1,1)^n$, and $\nabla \mathcal F(\m_*)=0$ for $\m_*\in Q$, it follows that 
\[
    \F(\m)-\F(\m_*)\geq \frac{c\|\m-\m_*\|_2^2}{2} \, .
\]
for all $\m\in Q$.
\end{lemma}

\begin{proof}

% We recall the equivalence of \cite[Definition 12]{lu2018relatively} with the strict convexity condition
% \[
% c \Dm \preceq \nabla^2 \mathcal F(\m).
% \] 
Using \eqref{eq:alt-strongly-convex} and \eqref{eq:Bregman-bigger}, and observing that $\nabla \mathcal F(\m_*)=0$, we obtain
\[
    \frac{\mathcal F(\m)-\mathcal F(\m_*)}{\|\m-\m_*\|_2^2}
    \geq
    \frac{\mathcal F(\m)-\mathcal F(\m_*)}{2\cdot\Dh(\m,\m_*)}
    \geq 
    \frac{c}{2} \, .
\]
\end{proof}

\begin{lemma}
\label{lem:Boundary}
Suppose $\F:Q_*\to \mathbb R$ is $c$-strongly convex in the convex set $Q_* :=B(\bm,\rho)\cap (-1,1)^n$.
If $\bx_*\in \partial Q_*$, $x_{*,k} =+1$ (respectively, $x_{*,k}=-1$) and
 $|x_j|<1$ for all $j\in[n]\setminus\{k\}$,
then $\lim_{t\to 0+}\partial_{x_k}\F(\bx_*-t\bfe_k) = +\infty$
(respectively $\lim_{t\to 0+}\partial_{x_k}\F(\bx_*+t\bfe_k) = -\infty$.)
\end{lemma}
\begin{proof}
Consider the case $x_k=+1$ (as the case $x_k=-1$ follows by symmetry.)
Then there exists $t_0>0$ such that $\bx_*-t\bfe_k\in Q_*$ for all $t\in (0,t_0]$.
 Let $\bx(s) := \bx_*-(t_0-s)\bfe_k$, $s\in [0,t_0)$.
 Then 
 \begin{align*}
 \partial_{x_k}F(\bx(s))&= \partial_{x_k}F(\bx(0))+\int_0^s \partial_{x_k}^2F(\bx(u))\, \de u\\
 & = \partial_{x_k}F(\bx(0))+\int_0^s \<\bfe_k,\nabla^2F(\bx(u))\bfe_k\>\, \de u\\
 & \ge \partial_{x_k}F(\bx(0))+c\int_0^s (1-x_k(u)^2)^{-1}\, \de u\\
 & \ge \partial_{x_k}F(\bx(0))+c \int_0^s (1-(1-t_0+u)^2)^{-1}\, \de u, .
 \end{align*}
 The last integral diverges as $s\uparrow t_0$, thus proving the claim.
\end{proof}

\begin{lemma}
\label{lem:approx-stationary-pt-to-minimizer}
Suppose $\F:Q\to \mathbb R$ is $c$-strongly convex for a convex set $Q\subseteq (-1,1)^n$. Moreover suppose that 
\[
    \|\nabla \mathcal F(\m)\|\leq c\sqrt{\eps n}
\]
for some $\m\in Q$ with 
\[
    B\left(\m,2\sqrt{\eps n}\right)\cap (-1,1)^n\subseteq Q \, .
\]
Then there exists a unique
$\m_*\in B\left(\m,2\sqrt{\eps n}\right)\cap (-1,1)^n$ satisfying $\nabla \F(\m_*)=0$, which 
is in fact a global minimizer of $\F$ on $Q$. Moreover
\begin{equation}
\label{eq:approx-stationary-approx-opt}
    \F(\m)-\F(\m_*)\leq 2c \eps n \, .
\end{equation}

\end{lemma}

\begin{proof}
Let $Q_\le :=\{\bx\in Q: \F(\bx)\le \F(\m) \}$. Then, for any $\bx\in Q_0$, we have
\begin{align*}
0&\ge \F(\bx) -\F(\m)\\
& \ge -c\sqrt{\eps n}\|\bx-\m\|_2+c  \Dh(\bx;\m)\\
& \ge  -c\sqrt{\eps n}\|\bx-\m\|_2+\frac{c}{2}\|\bx-\m\|_2^2\, .
\end{align*}
Hence $Q_\le \subseteq Q_* := B\left(\m,\sqrt{\eps n}\right)\cap (-1,1)^n$, $Q_*\subseteq Q$.
By  continuity three cases are possible: $(i)$~The minimum of $\F$ is
achieved in the interior of $Q_{\le}$; $(ii)$~The minimum is achieved along a sequence
$(\bx_i)_{i\ge 0}$, $\|\bx_i\|_{\infty}\to 1$; $(iii)$ the minimum is achieved at $\m_*\neq \m$
such that $\F(\m_*) = \F(\m)$. Case $(iii)$ cannot hold by strong convexity,
and case $(ii)$ cannot hold by Lemma \ref{lem:Boundary}.

Uniqueness of $\m_*$ follows by strong convexity, and 
$\nabla \F(\m_*)=0$ by differentiability.
Finally
\[
    \F(\m)-\F(\m_*)\leq \|\nabla \F(\m)\|\cdot \|\m-\m_*\|\leq 2c\eps n \, .
    \qedhere
\]
\end{proof}

\begin{lemma}

Suppose $\F:Q\to \mathbb R$ is relatively $c$-strongly convex. Let $\m_*$ be a local minimum of $\F$ belonging to the interior of $Q$, and suppose that $B\left(\m_*,2\sqrt{\eps n}\right)\cap (-1,1)^n\subseteq Q$. Consider for $\by\in\mathbb R^n$ the function
\[
    \F_{\by}(\m)=\F(\m)-\langle\by,\m\rangle.
\] 
Then $\F_{\by}$ is relatively $c$-strongly convex on $Q$ for any $\by\in \mathbb R^n$. 
If $\|\by\|\leq (c/2)\sqrt{\eps n}$, then $\F_{\by}$ has a unique stationary point and minimizer $\m_*(\by) \in Q$. Moreover if $\|\by\|,\|\hby\|\leq \frac{c\sqrt{\eps n}}{2}$ then
\begin{equation}
\label{eq:local-lip-minimizer}
    \|\m_*(\by)-\m_*(\hby)\| \leq \frac{\|\by-\hby\|}{c}.
\end{equation}
\end{lemma}

\begin{proof}

The relative $c$-strong convexity of $\F_{\by}$ is clear as the Hessian of $\F_{\by}$ 
does not depend on $\by$. For $\|\by\|\leq (c/2)\sqrt{\eps n}$, because
\[
    \|\nabla \F_{\by}(\m_*)\|= \|\by\| \le \frac{c\sqrt{\eps n}}{2} \quad\text{ and }\quad B\left(\m_*,\sqrt{\eps n}\right)\cap (-1,1)^n\subseteq Q \, ,
\]
Lemma~\ref{lem:approx-stationary-pt-to-minimizer} implies the existence of a unique minimizer
\[
    \m_*(\by)\in B\left(\m_*,\sqrt{\eps n}\right)\cap (-1,1)^n\subseteq Q \, .
\]
If $\|\hby\|\leq (c/2)\sqrt{\eps n}$ also holds, $\F_{\hby}$ is $c$-strongly convex on 
\[
    B\left(\m_*(\hby),\sqrt{\eps n}\right)\cap (-1,1)^n\subseteq B\left(\m_*,2\sqrt{\eps n}\right)\cap (-1,1)^n \subseteq Q.
\]
Moreover since $\|\by-\hby\|\leq c\sqrt{\eps n}$, we obtain
\begin{align*}
    \|\nabla \F_{\hby}(\m_*(\by))\| 
    & = 
    \|\by-\hby\|
    = 
    c\sqrt{\eps' n} \, ,
\end{align*}
for $\eps' = \frac{\|\by-\hby\|^2}{c^2 n}\leq \eps$. Therefore the conditions of Lemma~\ref{lem:approx-stationary-pt-to-minimizer} are satisfied with $(\F_{\hby},\m_*(\by),\eps')$ in place of $(\F,\m,\eps)$. Equation~\eqref{eq:local-lip-minimizer} now follows since 
\[
    \|\m_*(\by)-\m_*(\hby)\|\leq \sqrt{\eps' n} = \frac{\|\by-\hby\|}{c} \, .
    \qedhere
\]
\end{proof}

We now analyze the convergence of Algorithm~\ref{alg:NGD} from a good initialization.
\begin{lemma}
\label{lem:NGD-convergence-general}
Suppose $\F(\cdot)=\cuF_{\sTAP}(\,\cdot\,;\by,q_K(\beta,t))$ has a local minimum at $\m_*$ and 
is relatively $c$-strongly-convex on $B(\m_*,\sqrt{\eps n})\cap (-1,1)^n$, and also $C$-relatively
 smooth on $(-1,1)^n$. Suppose 
\begin{align}
\label{eq:NGD-Tech}
    \hm^0\in B\left(\m_*,\sqrt{\eps n}\right)\cap (-1,1)^n
\end{align}
satisfies
\begin{equation}
\label{eq:good-init-NGD}
\F(\hm^0)<\F(\m_*)+\frac{c\eps n}{8}.
\end{equation}
Then there exist constants $\eta_0,C'>0$ depending only on $(C,c,\eps)$ such that the following holds.
 If Algorithm~\ref{alg:NGD} is initialized at $\hm^0$ with learning rate $\eta=1/L\in (0,\eta_0)$, 
 then, for every $K\geq 1$ 
\begin{align}
\label{eq:NGD-good-value}
    \F(\hm^K) 
    & \leq 
    \F(\m_*)+C'n\left(1+\frac{\|\atanh(\hm^0)\|_2}{\sqrt{n}}\right)(1-c\eta)^K,\\
\label{eq:NGD-good-approx}
    \|\hm^K-\m_*\|_2 
    & \leq 
    C'\sqrt{n}\left(1+\frac{\|\atanh(\hm^0)\|_2}{\sqrt{n}}\right)(1-c\eta)^{K/2}.
\end{align}
\end{lemma}

\begin{proof}
Recall Eq.~\eqref{eq:NGD}, which we copy here for the reader's convenience:
\begin{equation}
\label{eq:NGD-App}
    \hm^{i+1}=\argmin_{\bx\in (-1,1)^n} \big\langle \nabla \F(\hm^i),\bx-\hm^{i} \big\rangle + L\cdot \Dh(\bx,\hm^{i}).
\end{equation} 

If $\eta_0\leq \frac{1}{2C}$ then \cite[Lemma 3.1]{lu2018relatively} applied to the linear
 (hence convex) function $\langle \nabla\F(\hm^i),\,\cdot\,\rangle$ states that for all
  $\m\in (-1,1)^n$,
\begin{equation}
\label{eq:3-point}
    \langle \nabla\F(\hm^i),\hm^{i+1}\rangle + L\Dh(\hm^{i+1},\hm^i)+L\Dh(\m,\hm^{i+1})\leq \langle  \nabla \F(\hm^i),\m\rangle + L\Dh(\m,\hm^i).
\end{equation}
Moreover the global relative smoothness shown in \eqref{eq:global-smooth} implies that for $\m,\m'\in (-1,1)^n$,
\begin{equation}
\label{eq:global-smooth-bregman}
\F(\m)\leq \F(\m')+\langle\nabla\F(\m'),\m-\m'\rangle+ C\cdot\Dh(\m,\m').
\end{equation}
Combining  Eqs.~\eqref{eq:3-point} and \eqref{eq:global-smooth-bregman} yields
\begin{equation}
\begin{aligned}
\label{eq:weird-ineq}
\F(\hm^{i+1}) & \leq \F(\hm^i)+\langle \nabla \F(\hm^i),\hm^{i+1}-\hm^i\rangle + L\Dh(\hm^{i+1},\hm^i)\\
&\leq \F(\hm^i)+\langle \nabla \F(\hm^i),\m-\hm^i\rangle + L\Dh(\m,\hm^i)-L\Dh(\m,\hm^{i+1}).
\end{aligned}
\end{equation}
Setting $\m=\hm^i$, we find
\[
    \F(\hm^{i+1})\leq \F(\hm^i),\quad \forall ~i\in [K].
\] 
We next prove by induction that for each $i\geq 1$, 
\begin{equation}
\label{eq:continue-good-NGD}
\F(\hm^i)<\F(\m_*)+\frac{c\eps n}{8},
\quad\quad 
\|\hm^i-\m_*\|<\sqrt{\eps n}.
\end{equation}
The base case $i=0$ holds by assumption. Suppose \eqref{eq:continue-good-NGD} holds for $i$. It follows that
\[
    \F(\hm^{i+1})\leq \F(\hm^i)\leq \F(\m_*)+\frac{c\eps n}{8}.
\] 
In fact, local $c$-strong convexity
\[
    \nabla^2 \F(\m) \succeq c\Dm \succeq c\bI_n,\quad \m\in B(\m_*,\sqrt{\eps n})\cap (-1,1)^n
\]
implies $\hm^i$ is even closer to $\m_*$ than required by \eqref{eq:continue-good-NGD}:
\[
    \|\hm^i - \m_*\|_2 \leq \sqrt{\frac{\F(\hm^i)-\F(\m_*)}{c}}\leq \frac{\sqrt{\eps n}}{2}.
\]
Next we bound the movement from a single NGD step. Comparing values of \eqref{eq:NGD-App} at $\hm^i$ and the minimizer $\hm^{i+1}$ implies
\begin{equation}
\label{eq:another-ineq}
    \langle \nabla \F(\hm^i),\hm^{i+1}-\hm^i\rangle + L\Dh(\hm^{i+1},\hm^i)\leq 0.
\end{equation}
From definition of Bregman divergence and the fact that (on the high probability event 
$\|\bA\|_{\text{op}}\le 3$) $\|\nabla\F+\nabla \bh\|_2\le C\sqrt{n}$ 
 (thanks to the special form of $\F(\,\cdot\,)=\cuF_{\sTAP}(\,\cdot\,;\by,q_K(\beta,t))$,
\begin{align*}
    |\langle \nabla \F(\hm^i),\hm^{i+1}-\hm^i\rangle + \Dh(\hm^{i+1},\hm^i)| 
    & = 
    |\langle \nabla \F(\hm^i)+\nabla \bh(\hm^i),\hm^{i+1}-\hm^i\rangle - \bh(\hm^{i+1})+\bh(\hm^i)|\\
    & \leq C_1 n\left(1+\frac{\|\hm^{i+1}-\hm^i\|}{\sqrt{n}}\right).  
\end{align*}
Moreover assuming $L>1$, \eqref{eq:Bregman-bigger} implies
\[
    (L-1) \Dh(\hm^{i+1},\hm^i)\geq \frac{L-1}{2} \|\hm^{i+1}-\hm^i\|^2.
\]
Substituting the previous two displays into \eqref{eq:another-ineq} yields
\[
    0\geq \frac{L-1}{2} \|\hm^{i+1}-\hm^i\|^2 - C_2\sqrt{n} \|\hm^{i+1}-\hm^i\|_2 - C_2n
\]
and so
\[
    \|\hm^{i+1}-\hm^i\|_2\leq \frac{C_3\sqrt{n}}{\sqrt{L-1}}.
\]
Taking $L$ large enough, it follows that 
\[
    \|\hm^{i+1}-\m_*\|\leq \|\hm^{i+1}-\hm^i\|_2 + \|\hm^{i}-\m_*\|_2\leq \sqrt{\eps n}.
\]
This completes the inductive proof of Eq.~\eqref{eq:continue-good-NGD}, which we now use to 
analyze convergence of Algorithm~\ref{alg:NGD}. Indeed from the first part of \eqref{eq:continue-good-NGD}, the local relative strong convexity of $\F$ implies
\[
    \F(\hm^i)+\langle \nabla\F(\hm^i),\m_*-\hm^i\rangle\leq \F(\m_*)-c\Dh(\m_*,\hm^i),\quad\quad \forall~i\in [K].
\]
Setting $\m=\m_*$ in \eqref{eq:weird-ineq} and combining yields
\[
    \F(\hm^{i+1})\leq \F(\m_*)+(L-c)\Dh(\m_*,\hm^i)-L\Dh(\m_*,\hm^{i+1}).
\]
Multiplying by $\left(\frac{L}{L-c}\right)^{i+1}$ and summing over $i$ gives 
\[
    \sum_{i=0}^{K-1} \left(\frac{L}{L-c}\right)^{i+1}\F(\hm^{i+1})\leq \sum_{i=0}^{K-1} \left(\frac{L}{L-c}\right)^{i+1} \F(\m_*)+L\Dh(\m_*,\hm^0).
\]
Since the values $\F(\hm^i)$ are decreasing, we find
\begin{align*}
    \F(\hm^K) 
    & \leq 
    \F(\m_*)+L\left(\sum_{i=0}^{K-1}\left(\frac{L}{L-c}\right)^{i+1}\right)^{-1} \Dh(\m_*,\hm^0)\\
    & \leq 
    \F(\m_*)+L\left(1-c\eta\right)^{K} \Dh(\m_*,\hm^0).
\end{align*}
Using Eq.~\eqref{eq:Bregman-bound} together with the last display
proves Eq.~\eqref{eq:NGD-good-value}.

 It was shown above by induction that $\hm^K$ is in a $c$-strongly convex neighborhood of $\m_*$. 
 Using strong convexity in Euclidean norm yields
\[
    \|\hm^k-\m_*\|\leq \sqrt{\frac{\F(\hm^K)-\F(\m_*)}{c}}
\]
and so \eqref{eq:NGD-good-approx} follows as well. 
\end{proof}

\begin{lemma}
Assume $\|\bA\|_{op}\leq 3$. For any $\m,\n\in (-1,1)^n$, and $\by,\hby\in\mathbb R^n$, and $q\in [0,1]:$
\begin{equation}
    \label{eq:grad-F-atanh}
    \|\nabla \cuF_{\sTAP}(\m,\by,q)-\nabla \cuF_{\sTAP}(\n,\hby,q)\| 
    \leq 
    (4\beta^2+4)\|\atanh(\m)-\atanh(\n)\| + \|\by-\hby\|.
    % \\
    % \label{eq:F-atanh}
    % \|\F_{\sTAP}(\m,\by;q)-\F_{\sTAP}(\n,\hby;q)\|
    % &\leq \big((4\beta^2+4) n^{1/2} + \|\by\|\big)\|\atanh(\m)-\atanh(\n)\|+n^{1/2}\|\by-\hby\|.
\end{equation}
\end{lemma}

% \mscomment{Second part doesn't seem to be used. Leaving for now just in case.}

\begin{proof}
The inequality~\eqref{eq:grad-F-atanh} follows with the smaller constant factor $\beta^2+3\beta+1\leq 4\beta^2+4$ using \eqref{eq:grad-F} and the fact that $\tanh(\cdot)$ is $1$-Lipschitz. 

% To show \eqref{eq:F-atanh}, we observe that the function $f(x)=h(\tanh(x))$ is also $1$-Lipschitz. Indeed, $h'(t)=-\atanh(t)$ and so
% \[
%     f'(x)=-\atanh(\tanh(x))\cdot \tanh'(x) \in [-1,1].
% \]
% Using this observation to bound the entropy term difference in $\F_{\sTAP}$ gives
% \begin{align*}
%     \|\F_{\sTAP}(\m,\by;q)-\F_{\sTAP}(\n,\by;q)\|
%     &\leq \frac{\beta}{2} \big|\langle \m,\bA\m\rangle - \langle \n,\bA\n\rangle\big| + \|\by\|\cdot \|\m-\n\| \\
%     &\quad 
%     +\|\atanh(\m)-\atanh(\n)\| +\frac{\beta^2}{2} \big|\langle \m,\m\rangle -\langle \n,\n\rangle\big|.
% \end{align*}
% For any $\bA$, we have
% \begin{align*}
%     \big|\langle \m,\bA\m\rangle - \langle \n,\bA\n\rangle\big| 
%     & \leq \big|\langle \m,\bA\m\rangle - \langle \n,\bA\m\rangle\big| + \big|\langle \n,\bA\m\rangle - \langle \n,\bA\n\rangle\big| 
%     \\
%     & \leq 2n^{1/2}\|\bA\|_{op}\cdot \|\m-\n\|_2.
% \end{align*}
% Applying this estimate for the true choice of $\bA$ as well as for $\bA=\bI_n$, we obtain 
% \begin{align*}
%     \|\F_{\sTAP}(\m,\by;q)-\F_{\sTAP}(\n,\by;q)\|
%     &\leq \big((\beta^2+3\beta+1) n^{1/2} + \|\by\|\big)\|\atanh(\m)-\atanh(\n)\|.
% \end{align*}
% Of course,
% \[
%     \|\F_{\sTAP}(\n,\by;q)-\F_{\sTAP}(\n,\hby;q)\|\leq n^{1/2}\|\by-\hby\|.
% \]
% Combining implies \eqref{eq:F-atanh}.
\end{proof}

% Finally we give a simple apriori estimate for $\|\by\|$.

% \mscomment{This also seems not to be used. }

% \begin{lemma}

% For $(\bA,\by)\sim \mathbb P^{\rd}_t$, we have
% \[
%     \plim_{n\to\infty}\frac{\|\by\|}{\sqrt{n}}=\sqrt{t^2+t}.
% \]

% \end{lemma}

% \begin{proof}

% By contiguity, it suffices to show the result for the planted model. Taking $\by=t\m+\bW(t)$ as in \eqref{eq:sdey}, we write
% \[
%     \|\by\|_2^2 = \langle \by,\by\rangle = t^2n + 2t\langle \m,\bW(t)\rangle + \|\bW(t)\|^2_2.
% \]
% We have
% \[
%     \plim_{n\to\infty} \|\bW(t)\|^2_2 = t
% \]
% by the law of large numbers applied across the $n$ coordinates, while $\langle \m,\bW(t)\rangle$ has mean $0$ and variance $O(n)$, hence 
% \[
%     \plim_{n\to\infty} \frac{\langle \m,\bW(t)\rangle}{n}=0.
% \]
% Combining implies the result. 

% \end{proof}

\subsection{Proof of Lemma~\ref{lem:local-landscape}}
\label{app:local-landscape}

We split the proof into four parts.

\begin{proof}[Proof of Lemma~\ref{lem:local-landscape}, Part~\ref{it:landscape-basic}]
Fix $c = (1/4)-(\beta/2)>0$.
Lemma~\ref{lem:TAP-stationary} implies that for $K_{\sAMP}=K_{\sAMP}(\beta,\T,\eps)$ sufficiently 
large, we have with probability $1-o_n(1)$ 
\begin{align}
    \|\nabla \cuF_{\sTAP}(\hm^{\sAMP};\by,q_*)\|\leq \frac{c\sqrt{\eps t n}}{4}\, ,
    \label{eq:again-small-grad}\\
    \hm^{\sAMP}:=\AMP(\bA,\by(t);K_{\sAMP}),\;\; q_*:=q_*(\beta,t)\, .\nonumber
\end{align}
Therefore, if $\|\by(t)-\hby\|\leq (c\sqrt{\eps tn})/4$ then
\begin{align*}
    \|\nabla \cuF_{\sTAP}(\hm^{\sAMP};\hby,q_*)\|\leq 
    \|\nabla \cuF_{\sTAP}(\hm^{\sAMP};\by(t),q_*)\|+\|\by-\hby\|\leq \frac{c}{2}\sqrt{\eps t}\, 
\end{align*}
Moreover Lemma~\ref{lem:local-convex} implies that there exist $\eps_0$, $c>0$
such that for all $\eps\in (0,\eps_0)$, 
\[
    \nabla^2 \cuF_{\sTAP}(\m;\hby,q_*)
    =
    \nabla^2 \cuF_{\sTAP}(\m;\by(t),q_*)
    \succeq 
    c\Dm,\quad\quad \forall~\m\in B\left(\hm^{\sAMP},\sqrt{\eps tn}\right)\cap (-1,1)^n.
\]
Using $\eps t/4$ in place of $\eps$ in Lemma~\ref{lem:approx-stationary-pt-to-minimizer}, 
it follows that there exists a local minimum
\[
    \m_*(\bA,\hby;q_*)\in B\left(\hm^{\sAMP},\frac{\sqrt{\eps tn}}{2}\right)\cap (-1,1)^n
\]
of $\cuF_{\sTAP}(\,\cdot\, ,\hby;q_*)$ which is 
also the unique stationary point in $B\left(\hm^{\sAMP},(1/2)\sqrt{\eps tn}\right)\cap (-1,1)^n$.

We next claim that, for any $K>K_{\sAMP}$, with probability $1-o_n(1)$, this local minimum 
is also the unique stationary point in 
$B\left(\AMP(\bA,\by(t);k),(1/2)\sqrt{\eps tn}\right)\cap (-1,1)^n$.
Indeed for $K_{\sAMP}$ sufficiently large (writing for simplicity $\by=\by(t)$):
\begin{align*}
    \plim_{n\to\infty}\sup_{k_1,k_2\in [K_{\sAMP},K]} \|\AMP(\bA,\by;k_1)-\AMP(\bA,\by;k_2)\|^2 
    &=
    \sup_{k_1,k_2\in [\kalg,K]} \plim_{n\to\infty} \|\AMP_{\beta}(\bA,\by;k_1)-\AMP_{\beta}(\bA,\by;k_2)\|^2\\
    & \leq n\cdot \sup_{k_1,k_2\geq K_{\sAMP}} |q_{k_1}(\beta,t) - q_{k_2}(\beta,t)|.
\end{align*}
From Eq.~\eqref{eq:uniform-gamma-limit}, by eventually increasing $K_{\sAMP}$, we have 
\[
    \sup_{k_1,k_2\geq K_{\sAMP}}|q_{k_1}(\beta,t) - q_{k_2}(\beta,t)|\leq \frac{\eps t}{16}.
\]
For such $K_{\sAMP}$, with probability $1-o_n(1)$, all $k\in [K_{\sAMP},K]$ satisfy
\begin{align*}
    \|\m_*(\bA,\by;q_{K_{\sAMP}}) - \AMP(\bA,\by;k)\|
    & \leq 
    \|\m_*(\bA,\by;q_{\sAMP}) - \AMP(\bA,\by;K_{\sAMP})\|
    \\
    &
    \quad\quad
    +
    \|\AMP(\bA,\by;k) - \AMP(\bA,\by;K_{\sAMP})\|\\
    & \leq
    \frac{\sqrt {\eps t n}}{2}+\sqrt{\frac{\eps t n}{4}}\\
    & \le \frac{3}{4}\sqrt{\eps t n}.
\end{align*}
Let 
\[
    S(k,\rho) : = B\left(\AMP_{\beta}(\bA,\by;k),\rho\right)\cap (-1,1)^n\, ,\;\;\;
    \rho_{n,t}:=\sqrt{\eps nt}
\]
Recall that $\m_*(\bA,\by;q_*)$ is the unique stationary point of 
$\cuF_{\sTAP}(\, \cdot\, ;\by,q_*)$  in  $S(K_{\sAMP},\rho_{n,t})$.
By the above, it is also a stationary point in  $S(k,\rho_{n,t})$, for $k\in [K_{\sAMP},K]$.
Repeating the same argument as before, there is only one stationary point inside 
$S(k,\rho_{n,t})$, hence this must coincide with  $\m_*(\bA,\by;q_*)$.
\end{proof}

\begin{proof}[Proof of Lemma~\ref{lem:local-landscape}, Part~\ref{it:landscape-stationary-point-good}]

Because $K_{\sAMP}$ is large depending on $\delta_0$, Lemma~\ref{lem:TAP-stationary} implies that with 
probability $1-o_n(1)$,
\[
    \|\nabla \cuF_{\sTAP}(\AMP(\bA,\by;K_{\sAMP}),\by;q_*)\|\leq \frac{c\delta_0\sqrt{ t n}}{4}.
\]
Using $\frac{\delta_0\sqrt{t}}{4} $ in place of $\eps$ in 
Lemma~\ref{lem:approx-stationary-pt-to-minimizer}, it follows that the local minimizer
$\m_*(\bA,\by;q_*)$
of $\cuF_{\sTAP}(\,\cdot\,;\by,q_*)$ satisfies
\[
    \|\AMP(\bA,\by;K_{\sAMP})-\m_*(\bA,\by;q_*)\|\leq \frac{\delta_0\sqrt{tn}}{2}. 
\]
Since $K$ is sufficiently large depending on $\delta_0$, Lemma 
implies that with probability $1-o_n(1)$,
\[
    \|\m(\bA,\by)-\AMP(\bA,\by;K_{\sAMP})\|\leq \frac{\delta_0\sqrt{tn}}{2}.
\]
Combining, we obtain that with probability $1-o_n(1)$,
\begin{align*}
    \|\m(\bA,\by)-\m_*(\bA,\by;q_*)\| 
    & \leq \|\m(\bA,\by)-\AMP(\bA,\by;K_{\sAMP})\| + \|\AMP(\bA,\by;K_{\sAMP})-
    \m_*(\bA,\by;q_*)\|
    \\
    &\leq \delta_0\sqrt{tn}.
    \qedhere
\end{align*}
\end{proof}

\begin{proof}[Proof of Lemma~\ref{lem:local-landscape}, Part~\ref{it:landscape-lipschitz}]
The result is immediate from \eqref{eq:local-lip-minimizer}.
\end{proof}

\begin{proof}[Proof of Lemma~\ref{lem:local-landscape}, Part~\ref{it:landscape-NGD}]
We apply Lemma~\ref{lem:NGD-convergence-general} with $\F(\,\cdot\,)=\cuF_{\sTAP}(\,\cdot\, ;\hby,q_*)$ 
and $\m_*=\m_*(\bA,\hby;q_*))$ (with $q_* = q_*(\beta,t)$).
We need to check that assumptions  \eqref{eq:NGD-Tech}, \eqref{eq:good-init-NGD} of
Lemma~\ref{lem:NGD-convergence-general} hold for $\hm^0=\tanh(\bu^0)$ with $\bu^0$ satisfying 
Eq.~\eqref{ass:landscape-NGD}. 

To check assumption  \eqref{eq:NGD-Tech}, we take $K_{\sAMP}$ sufficiently large
and $\delta_0$ sufficiently small, obtaining
\begin{align*}
    \|\hm^0-\m_*(\bA,\hby;q_*)\|
    &\leq 
    \|\hm^0-\AMP(\bA,\by;K_{\sAMP})\|
    + \|\AMP(\bA,\by;K_{\sAMP}) -\m(\bA,\by) \|\\
    &\quad\quad
    + \|\m(\bA,\by) - \m_*(\bA,\by;q_*)\|
    + \|\m_*(\bA,\by;q_*) - \m_*(\bA,\hby;q_*)\|\\
    &\stackrel{(a)}{\leq} \frac{c\sqrt{\eps tn}}{96(\beta^2+1)} + \frac{1}{100}\sqrt{\eps tn} + 
    \delta_0\sqrt{tn} + \frac{\|\by-\hby\|}{c}\\
    &\leq
    \frac{\sqrt{\eps tn}}{3}
\end{align*}
 where  inequality $(a)$ holds
 with probability $1-o_n(1)$. In the last step we used $c\leq 1$.

To check Eq.~\eqref{eq:good-init-NGD}, 
we use \eqref{eq:grad-F-atanh} we find 
that with probability $1-o_n(1)$,
\begin{align*}
    \|\nabla\cuF_{\sTAP}(\hm^0;\hby,q_*)\|
    &\leq 
    \|\nabla\cuF_{\sTAP}(\AMP(\bA,\by;K_{\sAMP});\by,q_*)\| + \|\by-\hby\|\\
    &\quad\quad+(4\beta^2+4)\|\atanh(\hm^0)-\atanh(\AMP(\bA,\hby;K_{\sAMP}))\| \\
    &\leq 
    \|\nabla\cuF_{\sTAP}(\AMP(\bA,\by;K_{\sAMP});\by,q_*)\| 
    + \frac{c\sqrt{\eps t n}}{24}
    + \frac{c\sqrt{\eps t n}}{4}.
\end{align*}
Combining with Eq.~\eqref{eq:again-small-grad}, we find that with probability $1-o_n(1)$,
\[
     \|\nabla\cuF_{\sTAP}(\hm^0;\hby,q_*)\|\leq \frac{c\sqrt{\eps t n}}{6}.
\]
Finally, we apply 
 Lemma~\ref{lem:approx-stationary-pt-to-minimizer} with $\frac{\eps t}{9}$ in place of
 $\eps$, to get
\[
    \cuF_{\sTAP}(\hm^0;\hby,q_*)\leq \cuF_{\sTAP}(\m_*(\bA,\hby;q_*);\hby,q_*) + \frac{nc\eps t}{9}.
\]

Lemma~\ref{lem:NGD-convergence-general} now applies for $\eta_0$ sufficiently small. 
Moreover, with probability $1-o_n(1)$ the initialization $\bx^0$ satisfies
\begin{align*}
    \|\atanh(\hm^0)\| 
    &\leq  
    \|\atanh(\hm^0)-\atanh(\AMP(\bA,\by;K_{\sAMP}))\|+\|\atanh(\AMP(\bA,\by;K_{\sAMP}))\|
    \\
    &\leq
    \frac{c\sqrt{\eps tn}}{96(\beta^2+1)}+  \sqrt{3(\gamma_*(\beta,t)+t)}\sqrt{n}\\
    &\leq
    C(\beta,c,\T)\sqrt{tn}.
\end{align*}
Thus, \eqref{eq:NGD-good-approx} implies \eqref{eq:landscape-NGD-convergence} for a sufficiently 
large number $K_{\sNGD}$ of natural gradient iterations.
\end{proof}

\end{document}